\documentclass{amsart}

\usepackage[latin1]{inputenc}
\usepackage{amsfonts}
\usepackage{amsmath}
\usepackage{amsthm}
\usepackage{amssymb}
\usepackage{latexsym}
\usepackage{enumerate}

\newtheorem{theorem}{Theorem}[section]
\newtheorem{lemma}[theorem]{Lemma}
\newtheorem{proposition}[theorem]{Proposition}

\newtheorem{question}[theorem]{Question}

\newtheorem{definition}[theorem]{Definition}

\numberwithin{equation}{section}

\newcommand{\N}{\mathbb{N}}

\newcommand{\R}{\mathbb{R}}
\newcommand{\G}{\mathbb{G}}

\newcommand{\PP}{\mathbb{P}}

\newcommand{\SSS}{\mathbb{S}}
\newcommand{\forces}{\Vdash}

\newcommand{\F}{\mathcal{F}}
\newcommand{\FF}{\mathcal{F}}
\newcommand{\M}{\mathcal{M}}
\newcommand{\A}{\mathcal{A}}
\newcommand{\B}{\mathcal{B}}
\newcommand{\D}{\mathcal{D}}
\newcommand{\E}{\mathcal{E}}
\newcommand{\I}{\mathcal{I}}
\newcommand{\J}{\mathcal{J}}
\newcommand{\U}{\mathcal{U}}
\newcommand{\X}{\mathcal{X}}

\newcommand{\CC}{\mathcal{C}}

\newcommand{\bb}{\mathcal B(\ell_2)}
\newcommand{\HH}{\mathcal{H}}

\newcommand{\K}{\mathcal{K}(\ell_2)}

\newcommand{\QQ}{\mathcal{Q}(\ell_2)}

\begin{document}

\author{Piotr Koszmider}
\address{Institute of Mathematics of the Polish Academy of Sciences,
ul.  \'Sniadeckich 8,  00-656 Warszawa, Poland}
\email{\texttt{piotr.math@proton.me}}

\thanks{The  author was  partially supported by the NCN (National Science
Center, Poland) research grant no.\ 2020/37/B/ST1/02613.}

\subjclass[2020]{03E35, 46L05, 03E75}

\title[Small masas in the Cohen model]{Small masas of the Calkin algebra in the Cohen model}

\begin{abstract} 
We show that maximal abelian C*-subalgebras (masas) of the Calkin algebra
(the algebra of all bounded operators on the separable Hilbert space modulo compact operators)
may consistently have their densities strictly less than continuum and we 
describe many   isomorphism types of such masas.

Specifically, we prove that after adding any number of Cohen reals to a model of {\sf CH}
the algebra $C(K_\A)$ of all complex-valued continuous functions 
on the Stone space $K_\A$ of a Boolean algebra $\A$ of cardinality $\omega_1$ is 
$*$-isomorphic to a masa of the  Calkin algebra if and only if $\A$ does not admit
 a countably generated ultrafilter. Moreover, for every such Boolean algebra we obtain $\omega_2$ pairwise
unitarily non-equivalent such masas, none of which has a commutative lift.

We also show in {\sf ZFC} that if a C*-algebra of the form $C(K)$ for any 
compact Hausdorff $K$ is $*$-isomorphic to a masa of
the Calkin algebra, then no point of  $K$ may have character smaller than $\mathfrak p$.
Therefore, consistently, there may not be any masa of the Calkin algebra of density less than continuum.
\end{abstract}

\maketitle

\section{Introduction}

In this paper we consider the noncommutative 
analogue of the Boolean algebra $\wp(\N)/Fin$, namely the Calkin algebra 
$\QQ=\bb/\K$, where $\bb$ stands for the algebra of all bounded linear  operators
on the separable Hilbert space $\ell_2$ and $\K$ stands for its two-sided ideal
of compact operators (\cite{calkin}, cf. \cite{ilijas-book}). 
By $\pi: \bb\rightarrow \QQ$ we will denote the quotient homomorphism.
Like the Boolean algebra
$\wp(\N)/Fin$, the Calkin algebra is rather sensitive to additional set-theoretic hypotheses.

The most celebrated example of such phenomena, corresponding to
the problem of the existence of a nontrivial automorphism of $\wp(\N)/Fin$,
seems to be the undecidability of the existence of
outer automorphisms of $\QQ$ (\cite{phillips, farah-ann}).
Some other results with this flavour are presented, among others, in the
papers \cite{akemann-pure, pathology, farah-ma, damian, universal-calkin, masas, ss,
vaccaro-thesis, vaccaro-pacific, vaccaro-ijm, vignati-thesis}.
However, infinitary combinatorial arguments available in {\sf ZFC}
such as those related to constructions of almost disjoint families (\cite{masas}) or independent families
(\cite{pure}) are sometimes sufficient  in the context of the Calkin algebra as well.

Specifically,  we investigate here the class of maximal abelian C*-subalgebras of $\QQ$ known
as masas. We focus only on masas generated by projections, that is elements $P\in \QQ$ satisfying
$P=P^2=P^*$ (the reason is that presently available set-theoretic
 methods only work in this context, for other masas we
face $K$-theoretic obstacles).   
If a masa  is generated by projections, then, by the Gelfand-Naimark theorem,
its projections  form  a Boolean algebra $\A$  and  the masa is
$*$-isomorphic to $C(K_\A)$, where $K_\A$ is the Stone space of  $\A$ 
(Lemma \ref{cstar-ba-iso} cf. Definition \ref{ba-projections}).

 There are two  paradigmatic  types of masas of $\QQ$
due to the fact that they are of the form $\pi[\M]$, where $\M$ is a masa
of $\bb$. The first type occurs when $\M=\D(\E)$, where $\D(\E)$ is the C*-subalgebra of
all diagonal operators in $\bb$ with respect to an orthonormal basis $\E$ of $\ell_2$,
and the second occurs when $\M=\M(U)=\{U^*M_f U: f\in L_\infty([0,1])\}$,   where
$U: \ell_2\rightarrow L_2([0,1])$ is unitary and $M_f: L_2([0,1])\rightarrow L_2([0,1])$
are multiplication operators by $f\in L_\infty([0,1])$. All these  masas are generated by projections.

Masas of $\QQ$ of the form $\pi[\D(\E)]$ are $*$-isomorphic to $C(\N^*)$,
where $\N^*=\beta\N\setminus\N$  while
$\pi[\M(U)]$s are isomorphic to $C(K_\A)$, where $\A$ is the Lebesgue measure 
Boolean algebra. All other masas of $\QQ$ which have commutative lifts to $\bb$
are of these forms or are their direct sums
(see Chapter 12 of \cite{ilijas-book}, cf. Introduction to \cite{masas}).

The existence of other masas in $\QQ$ like those which are
not generated by projections or are generated by projections but do not have commutative
lifts was already considered by Brown, Douglas and Fillmore in \cite{bdf}  or
by Anderson in \cite{pathology} under the continuum hypothesis {\sf CH}, while Akemann and Weaver gave
a simple counting argument  showing in {\sf ZFC} the existence
of $2^{2^\omega}$ masas of $\QQ$ without commutative lifts. 
However, only in \cite{masas} were the isomorphism types of such masas described,
some in {\sf ZFC} and all of the form $C(K_\A)$ for $\A$ a Boolean algebra under {\sf CH}.

The {\sf CH} characterization of $\A$s such that $C(K_\A)$ is $*$-isomorphic to some masa of $\QQ$
obtained in \cite{masas} cannot hold in {\sf ZFC} as witnessed, 
among others, by the Cohen model  (\cite{vaccaro-thesis},  \cite{damian}). So it is natural to 
continue this investigation of the class of Boolean algebras $\A$ for which $C(K_\A)$s 
can be masas of $\QQ$ specifically in the Cohen model or in general aiming at alternative characterizations
if possible in {\sf ZFC}.

If $C(K_\A)$ is $*$-isomorphic to a masa of $\QQ$, then necessarily $|\A|\leq2^\omega$.
 In \cite{masas} we also proved that
such a Boolean algebra cannot have a countably generated ultrafilter. The first result of this paper improves it
to:
\begin{theorem}\label{p-main}
If $K$ is a compact Hausdorff space such that
$C(K)$ is $*$-isomorphic to a masa in $\QQ$, then $K$ has no point of 
character\footnote{By a character $\chi(x)$ of a point $x$ in a topological space (ultrafilter of a Boolean algebra)
we mean the minimal cardinality of a basis of neighbourhoods at this point (a basis of ultrafilter).
The minimal character $\min\chi(K)$ ($\min\chi(\A)$) of a compact space $K$ (a Boolean algebra $\A$) is 
the minimum of the characters of all points of $K$ (all ultrafilters of $\A$).} less than
 $\mathfrak p$\footnote{The pseudointersection number $\mathfrak p$
is the least cardinality of a family $\mathcal F$ of infinite subsets of $\N$
such that every nonempty finite intersection of members of $\mathcal F$ is infinite
and there is no infinite $X\subseteq\N$ such that $X\setminus F$ is finite for every
$F\in\mathcal F$; see Section 9.6 of \cite{ilijas-book}.}.
\end{theorem}
\begin{proof}
Suppose that $x\in K$ has character $\kappa<\mathfrak p$.
For finite $\kappa$ we use Proposition 4.1 of \cite{masas}. For infinite $\kappa$
we  use Proposition \ref{masigma} and a theorem of M. Bell (\cite{bell}) in a version of Theorem 3.1 from
\cite{velickovic} which says that $\kappa<\mathfrak p$ is equivalent to
${\sf MA}_{\sigma-\hbox{\rm centered}}(\kappa)$. 
\end{proof}
So, for example Martin's axiom implies that there are no masas of $\QQ$ of density strictly less
than continuum. 
However, our second main result shows that one can consistently 
have many diverse masas of the form $C(K_\A)$
of density strictly less than $2^\omega$ and 
such that no ultrafilter of $\A$ has character bigger than $\mathfrak p<2^\omega$.

\begin{theorem}\label{main}
In the model obtained by adding Cohen reals to a model of {\sf CH} the following holds: 
For every Boolean algebra $\A$ of cardinality $\omega_1$ which does
not admit a countably generated ultrafilter
there is a family $\{\CC_\alpha: \alpha<\omega_2\}$
of unitarily non-equivalent 
 C*-subalgebras of $\QQ$ such that
  for every $\alpha<\omega_2$ the algebra
$\CC_\alpha$ is a masa of $\QQ$ without a commutative lift
which is $*$-isomorphic to $C(K_\A)$.
\end{theorem}
\begin{proof} If the model is obtained by adding at least $\omega_2$ Cohen reals
to a model of {\sf CH}, then Proposition \ref{cohen-proposition} directly implies the theorem.
If we add fewer than $\omega_2$ Cohen reals, we obtain a model of {\sf CH} so
the statement of the theorem follows from the results of \cite{masas} under {\sf CH}.
\end{proof}

These results bring us closer to understanding which abelian C*-algebras may be masas of
the Calkin algebra. 
Recall that in \cite{masas} we also proved that,  assuming {\sf CH},  for Boolean algebras $\A$  of cardinality $2^\omega$
the hypothesis that the minimal character of $\A$ is uncountable is necessary and sufficient for
$C(K_\A)$ to be $*$-isomorphic to  a masa of $\QQ$. 
Theorem \ref{main} shows that  a similar situation takes place for
Boolean algebras of cardinality $\omega_1$ in the Cohen model. However this similarity stops at higher
cardinals. By a result of Vaccaro from \cite{vaccaro-thesis} the Calkin algebra in this
model cannot contain well-ordered chains of projections of length $\omega_2$,
so, for example, the algebras  $C([0,\omega_2]\times \{0,1\}^\kappa)$ cannot be masas there
for any uncountable cardinal $\kappa$  while the corresponding Boolean algebras have no ultrafilters of countable
character. Similarly by the results of \cite{damian}
the Calkin algebra in the Cohen
model cannot contain the C*-algebra $\ell_\infty(c_0(\omega_2))$, so
also C*-algebras like $C(\beta(\N\times\N^*))$ cannot be masas of $\QQ$ there
while the corresponding Boolean algebra has all
ultrafilters of uncountable character.

We direct the reader to the last section of \cite{masas} to see the description
of $2^{2^\omega}$ nonisomorphic Boolean algebras $\A$ such
that $C(K_\A)$ is $*$-isomorphic to a masa of $\QQ$ in {\sf ZFC}. But here we would like
to inspect a bit the class $\mathcal C$ of Boolean algebras $\A$ in the Cohen model
provided by Theorem \ref{main}, i.e.,
of cardinality $\omega_1$ 
such that $C(K_\A)$ is $*$-isomorphic to a masa of $\QQ$.

In fact, there is a family of  $2^{\omega_1}$ pairwise nonisomorphic Boolean algebras in $\mathcal C$
and so there is a family of the same cardinality 
of  pairwise nonisomorphic masas of $\QQ$ corresponding to elements of $\mathcal C$.
 This follows,
for example, from the {\sf ZFC} construction of Geschke and Shelah in Theorem 3.1 of \cite{geschke-shelah}:
their algebras cannot have countably generated ultrafilters as
every countable family of nonzero elements of any of the algebras constructed
there is split by a sufficiently late independent generator (see (iii) in the proof of Lemma 3.2 in \cite{geschke-shelah}).
Also all Boolean algebras in $\mathcal C$ admit uncountable independent families
(Proposition \ref{ind}).  Definitely among 
elements of $\mathcal C$ there are algebras of the form  $Clop(\{0,1\}^{\omega_1}\times L)$,
where the weight of compact totally disconnected $L$ does not exceed $\omega_1$.

Many interesting examples of  algebras in $\mathcal C$ can be obtained using Proposition
\ref{reaping} by taking a Boolean algebra of cardinality, the minimal character and
reaping number all equal to $\omega_1$ and belonging to the ground model of {\sf CH}.
For example, one may take  $\wp(\N)/Fin\cap V$ or the ground model Lebesgue measure algebra.
It is also worth noting that 
$C(\{0,1\}^{\omega_1})$ which  is $*$-isomorphic to  a masa of $\QQ$ in the Cohen model,
and under {\sf CH} by \cite{masas},
 is the group C*-algebra of the discrete abelian group
$\bigoplus_{\omega_1}\mathbb Z_2$. This  is related to recent work
on lifting properties of uncountable abelian group C*-algebras  in \cite{miles}.

The methods of this paper mainly combine the methods of the paper \cite{masas} 
adapted to forcing extensions with classical methods for constructing objects which
are Cohen forcing indestructible (see e.g. Theorem VIII.2.3 of \cite{kunen}). However, the 
above combination needs to be delicate not only because of forcing extension technicalities
(e.g. members of the spectra of operators need not be in the ground model etc).

More specifically, the Cohen forcing quite easily may add countably generated ultrafilters
in Boolean algebras which do not admit them. By Theorem \ref{p-main} such an
ultrafilter of a ground model
Boolean algebra $\A$ shows that  $C(K_\A)$ is not $*$-isomorphic
 to a masa of $\QQ$. In fact, in the 
construction of our 
Cohen indestructible masa of the form $C(K_\A)$ 
we need the hypothesis  that  the Cohen forcing does not
add a countably generated ultrafilter to the (ground or intermediate model) 
Boolean algebra $\A$ (Proposition \ref{ch-proposition}).

This is a strong hypothesis, which implies in the Cohen model
for a Boolean algebra $\A$ of cardinality $\omega_1$ that it 
contains an uncountable independent family (Proposition \ref{ind}).
In fact the appearance of an independent element over
a given countable subalgebra  is evident in the proofs of Lemmas \ref{substitution}
and \ref{main-lemma} in the case when $\dot \U$ is a generic ultrafilter in a countable atomless Boolean algebra.

The structure of the paper is as follows. Section 2 gathers all necessary preliminaries. Many of them
are directly cited from the Preliminaries of \cite{masas}, where their simple and standard proofs
can be found. Section 3 contains the
proof of the main ingredient of Theorem \ref{p-main}. Section 4 includes all the ingredients
of the proof of Theorem \ref{main}: first preparatory but most technical Lemmas \ref{substitution}
and \ref{main-lemma},  later the construction of Cohen indestructible masas under
{\sf CH} in Proposition \ref{ch-proposition} and finally Proposition \ref{cohen-proposition}
which implements Proposition \ref{ch-proposition} in the Cohen model.

The Cohen model is the best understood model of {\sf ZFC} where the continuum hypothesis fails. 
However, most of this understanding does not concern noncommutative mathematics (see e.g. \cite{another}, \cite{juris-book}, \cite{juris-top}).
It seems that there  is much space for improvement of this situation.
The results of the paper generate the following natural  questions:
\begin{question}$ $
\begin{enumerate}
\item  For which Boolean algebras $\A$ of cardinality $\omega_2$ (or $2^\omega$)
does there exist a masa $*$-isomorphic to $C(K_\A)$ in Cohen models?
\item Is it true in the Cohen model that for a Boolean algebra $\A$ with
$\min\chi(\A)>\omega_1$ ($\min\chi(\A)=2^\omega$)  any two of the  following conditions are equivalent:
\begin{enumerate}
\item $\A$ is Boolean embeddable in $\wp(\N)/Fin$,
\item $C(K_\A)$ is  $*$-embeddable in $\QQ$,
\item $C(K_\A)$ is $*$-isomorphic to a masa of $\QQ$?
\end{enumerate}
\item Is there in {\sf ZFC} a masa of the Calkin algebra of the form $C(K_\A)$ such that 
$\A$ has an ultrafilter of character $\mathfrak p$? (Or can Theorem \ref{p-main}
be improved to another cardinal invariant?)
\end{enumerate}
\end{question}
We note that by the Parovi\v cenko theorem
(see e.g. \cite{handbook-van-mill}) and our Theorem \ref{main}  all  conditions
(a) - (c) above are true in the Cohen model for algebras of cardinality 
and the minimal character $\omega_1$.

\section{Preliminaries}

\subsection{Notation}

$1_X$ denotes the characteristic function of a set $X$. The restriction of a function $f$
to a set $X$  is denoted by $f|X$.  $\N$ stands for the set of all nonnegative integers.
$Fin$ denotes the ideal of finite subsets of $\N$.
If $(a_n)_{n\in \N}$ is a sequence of complex numbers and $X\subseteq \N$ is infinite, we say 
that $\lim_{n\in X}a_n=a$ if $\lim_{k\rightarrow \infty }a_{n_k}=a$,
where $(n_k)_{k\in \N}$ is the increasing enumeration of elements of $X$.
 $2^\omega$ stands for the cardinal of $\R$ and $\omega_1$ stands for
the first uncountable cardinal.

\subsection{Boolean algebras}

Boolean algebra terminology is standard and should follow e.g., the book \cite{koppelberg}.
We use   $\wedge, \vee, -$ for Boolean operations and $1, 0$ for the unit and zero of a Boolean algebra.
The Boolean order is defined as $A\leq B$ if $A\wedge B=A$. 
Filters are subsets of Boolean algebras which do not contain $0$
and  which are upward closed in the Boolean order and closed under $\wedge$.
Ultrafilters are maximal filters. If $\U$ is an ultrafilter of a Boolean algebra $\A$, then
$A\in \U$ or $1-A\in \U$ for every $A\in \A$.  

We will deal with homomorphisms of Boolean algebras
and in particular Boolean monomorphisms (homomorphisms whose kernels are $\{0\}$) and
Boolean isomorphisms.  We  use the adjective Boolean before
the noun homomorphism, monomorphism etc, to distinguish it from 
$*$-homomorphisms concerning C*-algebras.

For every Boolean algebra
$\A$ there is a compact Hausdorff totally disconnected (with open basis consisting of clopen sets) 
space $K_\A$ such that there exists a
Boolean isomorphism $h: \A\rightarrow Clop(K_\A)$, where $Clop(K)$ denotes the Boolean algebra
of clopen subsets of $K$ with $\cap, \cup$ and the complement as Boolean operations. 
$K_\A$ is called {\sl the Stone space} of $\A$. Its points correspond to
ultrafilters of $\A$. Namely, by the compactness of $K_\A$, every ultrafilter in $\A$ is of the form
$\U_x=\{h^{-1}(U): x\in U, U\in Clop(K_\A)\}$ for some $x\in K_\A$. 
In fact,  the Stone duality says that the category of compact Hausdorff totally disconnected spaces with continuous maps
is dual to the category of Boolean algebras with homomorphisms
(Chapter 3 of \cite{koppelberg}, Chapter 16.3 of \cite{semadeni}).

\begin{definition}\label{def-ba-split}
Suppose that $\A\subseteq \B$ are Boolean algebras, $B\in\B\setminus\A$
and $\U\subseteq \A$ is an ultrafilter in $\A$. We say that $B$ splits $\U$
if $A\wedge B\not =0\not= A\wedge (1-B)$ whenever $A\in \U$.
\end{definition}

\begin{definition}\label{def-ideal-simple} Suppose that $\A\subseteq \B$ are Boolean algebras and $B\in\B$.
We say that $\B$ is a simple extension of $\A$ by $B$ if $\B$ is generated by $\A\cup\{B\}$ and $\B\not=\A$.
We say that $\B$ is a simple extension of $\A$ if there is $B\in \B$ such that $\B$ is a simple extension of $\A$ by $B$.
When $\B$ is a simple extension of $\A$ by $B$ we define
$$\J_{\A,B}=\{A\in \A: B\wedge A\in \A\}.$$
\end{definition}

\begin{lemma}\label{ideal-split} Suppose that $\A\subseteq \B$ are Boolean algebras, $B\in\B$
and $\B$ is a simple extension of $\A$ by $B$. Then $\J_{\A,B}$ is a proper ideal of
$\A$ such that the following conditions are equivalent for every ultrafilter $\U$ of $\A$:
\begin{itemize}
\item $\U$ is split by $B$,
\item $\U\cap\J_{\A,B}=\emptyset$.
\end{itemize}
\end{lemma}
\begin{proof}
It is immediate that $\J_{\A,B}$ is an ideal. Since $B\not\in \A$, the ideal $\J_{\A,B}$ must be proper. Suppose that
$\U$ is an ultrafilter of $\A$. 

If there is $A\in \J_{\A,B}\cap \U$, then both $A\wedge B$ and
$A- B=A-(A\wedge B)$ belong to $\A$, and so,  one of them is in $\U$.
If $A\wedge B\in \U$, then $(A\wedge B)\wedge(1-B)=0$ showing that $\U$ is not split by $B$.
If $A-B\in \U$, then $(A-B)\wedge B=0$ showing that $\U$ is not split by $B$ as well.

If $\U$ is not split by $B$, then there is $A\in \U$ such that either $A\wedge B=0$ 
or $A\wedge  B=A$. Either case implies that $A\in \J_{\A,B}$. 
\end{proof}

\begin{lemma}\label{gdelta-split}
Suppose that $\B$ is a countable Boolean subalgebra of a Boolean 
algebra $\A$ which does not admit a countably
 generated ultrafilter. For every ultrafilter $\U$ of $\B$ there is $A\in\A\setminus \B$ which splits $\U$.
\end{lemma}
\begin{proof} If for every element $A\in \A\setminus \B$ we have $B\in \U$ such that
$B\leq A$ or $B\leq 1-A$, then $\U$ generates an ultrafilter in $\A$ which is countably
generated as $\B$ is countable.
\end{proof}

\begin{lemma}\label{sikorski} Let $\A$ be a Boolean algebra and $\B_1$, $\B_2$ be two
simple extensions of $\A$ by $B_1$ and $B_2$ respectively. Suppose that
\begin{enumerate} 
\item $\J_{\A, B_1}= \J_{\A, B_2}$ and
\item For every $A\in \J_{\A, B_2}$ we have $A\wedge B_1=A\wedge B_2$.
\end{enumerate}
Then there is a Boolean isomorphism $h: \B_1\rightarrow \B_2$
such that $h|\A=Id_{\A}$ and $h(B_1)=B_2$.
\end{lemma}
\begin{proof}
To construct $h: \B_1\rightarrow \B_2$
we will use Sikorski's extension criterion (Theorem 5.5 of \cite{koppelberg}, cf. 5.6 of \cite{koppelberg})
which implies that the identity on $\A$ can be extended to  an isomorphism $h$ as 
in the lemma if and only if the following two conditions hold for
all $A\in \A$:
\begin{enumerate}
\item[(i)] $A\wedge B_1=0 \Leftrightarrow A\wedge B_2=0$.
\item[(ii)] $A\wedge(1-B_1)=0 \Leftrightarrow A\wedge(1-B_2)=0$.
\end{enumerate}
First note that $A\wedge B_1=0$ or $A\wedge(1-B_1)=0$ means that 
$A\in \J_{\A, B_1}$.  Then (2) gives us the forward implications.
For the reverse implications, similarly note that
$A\wedge B_2=0$ or $A\wedge(1-B_2)=0$ means that 
$A\in \J_{\A, B_2}$ and again use (2).
\end{proof}

\subsection{C*-algebras}

Our terminology concerning C*-algebras should follow \cite{blackadar}, \cite{murphy} and \cite{ilijas-book}.
Our models of $n$-th dimensional Hilbert spaces are $\ell_2(n)$ and the model of
an infinite dimensional  separable Hilbert space is $\ell_2$.  If $F\subseteq\N$ by $\ell_2(F)$ we denote
the subspace of $\ell_2$ of $F$-supported vectors.
For a  closed subspace $V$ of a Hilbert space $V^\perp$ denotes
its orthogonal complement.
The commutator of $T, S$ from
a C*-algebra is denoted by $[T, S]=TS-ST$. The C*-algebra of all
bounded linear operators on $\ell_2$ ($\ell_2(n)$) is denoted  $\bb$ ($\mathcal B(\ell_2(n))$). 
The ideal of compact operators on $\ell_2$ is denoted $\K$. 
The Calkin algebra $\bb/\K$ is denoted $\QQ$. The quotient
map from $\bb$ onto $\QQ$ is denoted $\pi$. We will often be using 
the following two facts about compact operators:
\begin{itemize}
\item If $A\in \K$ and $(v_n)_{n\in \N}$ is a sequence of unit, pairwise orthogonal vectors in $\ell_2$,
then $\lim_{n\in \N}\|A(v_n)\|=0$.
\item If  $(v_n)_{n\in \N}$ is an orthonormal basis of $\ell_2$ and $A\in \bb$ satisfies
$\sum_{n\in \N}\|A(v_n)\|<\infty$, then $A\in \K$.
\end{itemize}
Both of them follow from the fact that an operator in $\bb$ is compact if 
and only if it can be norm approximated by finite-rank operators. 

If $\F\subseteq\QQ$, then 
$C^*(\F)$  denotes the C*-subalgebra of $\QQ$ generated by $\F$.
All C*-algebras we deal with will be unital. They will be C*-subalgebras
of $\bb$ or $\QQ$ or algebras of the form $C(K)$ of complex-valued
continuous functions on a compact Hausdorff space $K$.
By the Gelfand-Naimark theorem all commutative C*-algebras are $*$-isomorphic 
to some member of the latter class of C*-algebras.
Recall that a $*$-isomorphism of C*-algebras is an isometry (1.3.3 of \cite{ilijas-book}).
An element $S$ of a C*-algebra is self-adjoint (a projection) if $S=S^*$ ($S=S^2=S^*$).
Recall that if $S\in \QQ$ is self-adjoint (a projection), then there is $T\in \bb$ such that $\pi(T)=S$
and $T$ is self-adjoint ($T$ is a projection) (Lemma 2.5.4 and 3.1.13 of \cite{ilijas-book}).
The spectrum of an element $T$ of a unital C*-algebra is denoted by $\sigma(T)$.
Recall that it does not grow when we pass to a quotient,
in particular $\sigma(\pi(T))\subseteq \sigma(T)$ for $T\in \bb$  (Lemma 2.5.7 in \cite{ilijas-book}).

If $P$ is a projection in a unital C*-algebra $\A$, then
$P\A P=\{PTP:T\in\A\}$ is a C*-algebra with unit $P$, called the corner
of $\A$ determined by $P$ (Example 3.2.1 of \cite{murphy}).

If $\E=(e_n)_{n\in \N}$ is an orthonormal basis of $\ell_2$, then 
$\D(\E)$ denotes  the commutative C*-algebra of all operators diagonal with respect to $\E$, i.e.,
such operators that $e_n$ is their eigenvector  for every $n\in \N$. 
Weyl-von Neumann theorem asserts that for every self-adjoint $T\in \bb$
there is an orthonormal basis $\E$ and  $S\in \K$ such that $T+S\in \D(\E)$ (V. 4.1.2 (ii) of \cite{blackadar}).
In fact any countable set of commuting projections can be simultaneously diagonalized
modulo a compact operator (Lemma 12.4.4 of \cite{ilijas-book}). We will need a more specific version of this
in Lemma \ref{calkin-lift}.

\begin{lemma}[Lemma 3.1 of \cite{masas}]\label{commutator} 
Let $T\in \bb$ be 
self-adjoint and $c_1<c_2$ satisfy $c_1, c_2\in \sigma(T)$.
Let $\E=(e_n)_{n\in \N}$ be an orthonormal basis of $\ell_2$.  
Then for every $\varepsilon>0$ there is a finite $F\subseteq \N$,
and a projection  $P\in \bb$ 
 such that
$$\|[P_FTP_F, P]\|>{{c_2-c_1\over2}-\varepsilon}$$
and $\langle P(e_n), e_m\rangle\not=0$ implies $n, m\in F$ for every $n,m\in \N$,
where $P_F$ is the projection in $\bb$ onto the space spanned by $\{e_n: n\in F\}$.
\end{lemma}

\subsection{Boolean subalgebras of $\QQ$ and their liftings}

The set of projections  of a $C^*$ algebra $\A$ will be denoted
by $Proj(\A)$. All projections in C*-algebras of the form $C(K)$ for $K$ compact Hausdorff
are  characteristic functions $1_U$ for clopen $U\subseteq K$. The Boolean operations
in the Boolean algebra $Clop(K)$ can be translated to C*-algebraic operations on
the functions $1_U$ for $U\in Clop(K)$ as $1_{U\cap V}=1_U1_V$, $1_{U\cup V}=1_U+1_V-1_{U}1_V$
and $1_{K\setminus U}=1_K-1_U$. By the Gelfand-Naimark theorem this means that 
families of commuting projections in C*-algebras generate Boolean algebras with the
 above operations (see e.g. \cite{bade}, \S 30 of \cite{halmos}):

\begin{definition}\label{ba-projections} If $P, Q$ are commuting projections in a C*-algebra $\A$, then we define
\begin{itemize}
\item $P\wedge Q=PQ$,
\item $-P=I-P$,
\item $P\vee Q=P+Q-PQ$.
\end{itemize}
\end{definition}

Whenever we say that $h:\A\rightarrow \bb$ or $h:\A\rightarrow \QQ$  is a Boolean
homomorphism (monomorphism) we mean that the range of $h$ consists of commuting
projections and it is a Boolean algebra with the operations as in Definition  \ref{ba-projections} and
such that $h(1)$ is the unit of $\bb$ or $\QQ$ and $h(0)$ is the zero of $\bb$ or $\QQ$ respectively. 

\begin{lemma}[Lemma 2.6 of \cite{masas}]\label{cstar-ba-iso} Suppose that the C*-algebra $C^*(\A)$ is generated by a Boolean algebra
of  commuting projections $\A\subseteq\QQ.$  Then $C(K_\A)$ and $C^*(\A)$ are $*$-isomorphic, where $K_\A$
is the Stone space of $\A$.
\end{lemma}

\begin{lemma}[Lemma 2.9 of \cite{masas}]\label{calkin-lift}
Suppose that $\A\subseteq \QQ$ is a countable Boolean algebra of projections.
Then there are an orthonormal basis $\E$ of $\ell_2$ and a Boolean subalgebra
$\B$ of $\D(\E)$ such that
$$
\pi|\B:\B\longrightarrow\A
$$
is a Boolean isomorphism.
\end{lemma}

\begin{lemma}[Lemma 2.10 of \cite{masas}]\label{injective-Q} Suppose that $\A, \B$ are countable Boolean algebras with $\A\subseteq \B$.
Let $h: \A\rightarrow Proj(\QQ)$ be a Boolean monomorphism. Then there is
an extension $h'$ of $h$ to a Boolean monomorphism  $h':\B\rightarrow Proj(\QQ)$.
\end{lemma}

\begin{definition}\label{def-split}Suppose that $\A\subseteq \QQ$ is a Boolean algebra of projections
and $T\in \bb$ is self-adjoint such that $\pi(T)$ commutes with all elements of
$\A$. Let $\U$ be an ultrafilter of $\A$. We say that $T$ C*-splits $\U$ if there are
two distinct multiplicative states  $x_1, x_2$ of $C^*(\A\cup\pi(T))$  such that
$$\{A\in \A: x_1(A)=1\}=\{A\in \A: x_2(A)=1\}=\U.$$
\end{definition}

\begin{lemma}[Lemma 2.12 of \cite{masas}]\label{split-point} Suppose that $\A\subseteq \QQ$ is a Boolean algebra of projections
and $T\in \bb$ is self-adjoint such that $\pi(T)$ commutes with all elements of
$\A$ but $\pi(T)\not\in C^*(\A)$. Then there is an ultrafilter $\U$ 
 of $\A$ which is $C^*$-split by $T$.
\end{lemma}

\begin{lemma}[Lemma 2.13 of \cite{masas}]\label{c-split}Suppose that $\A\subseteq \QQ$
 is a Boolean algebra of projections
and $T\in \bb$ is self-adjoint such that $\pi(T)$ commutes with all elements of
$\A$. Suppose that  $\U$ is an ultrafilter of $\A$ which is C*-split by $T$.
Then there are reals $c_1<c_2$ such that
$$
c_1,c_2\in\sigma_{A\QQ A}(A\pi(T)A)
$$
for every $A\in\U$, where the spectrum is computed in the corner
$A\QQ A$, whose unit is $A$.
\end{lemma}

\subsection{Masas in $\bb$ and in $\QQ$}

A masa in a C*-algebra $\A$ is a maximal, with respect to inclusion,
commutative C*-subalgebra of $\A$. Two C*-subalgebras $\A$ and $\B$ of a unital C*-algebra $\CC$ are called
unitarily equivalent if there is a unitary $U\in\CC$ such that
$\B=U\A U^*$ (cf. Section 12.3 of \cite{ilijas-book}).

If $\A\subseteq\QQ$ is a C*-subalgebra, we say that $\A$ has a commutative lift to $\bb$
if there is a commutative C*-subalgebra $\B\subseteq\bb$ such that
$\pi[\B]=\A$ (cf. Section 12.4 of \cite{ilijas-book}).

We recall the description of masas of $\bb$ (see Proposition C.6.11 of \cite{ilijas-book}). If $\mu$ is a measure, then
$L_\infty(\mu)$ will be considered as the algebra of multiplication operators
on $L_2(\mu)$. A C*-subalgebra $\M\subseteq\bb$ is a masa of $\bb$ if and only if
there are a separable measure $\mu$ and a unitary operator
$$
U:\ell_2\longrightarrow L_2(\mu)
$$
such that
$$
\M=\{U^*M_fU:f\in L^\infty(\mu)\},
$$
where $M_f$ denotes the operator of multiplication by $f$ on $L_2(\mu)$.

Using the decomposition of a separable measure into its atomic and nonatomic
parts, every masa of $\bb$ is therefore unitarily equivalent to one of the
following forms:
\begin{enumerate}
\item[(1)] an atomic masa $\D(\E)$ for some orthonormal basis $\E$ of $\ell_2$;
\item[(2)] the nonatomic masa $L_\infty([0,1])$ acting by multiplication
on $L_2([0,1])$;
\item[(3)] a direct sum of an atomic masa and the nonatomic masa,
acting on an orthogonal decomposition
$$
\ell_2=\HH_a\oplus\HH_c.
$$
\end{enumerate}
Here the atomic summand may have any nonzero finite or countably infinite dimension.
In particular, every masa of $\bb$ is generated
by projections.

We will often use a theorem of Johnson and Parrott (see \cite{parrott}) which says that
if $\A\subseteq\bb$ is a masa of $\bb$, then $\pi[\A]$ is a masa of $\QQ$.
While constructing masas of $\QQ$ the following lemma will be very useful:

\begin{lemma}[Lemma 2.14 of \cite{masas}, cf. p. 72 of \cite{ss}]
\label{sa-masa}
If $\A$ is a commutative C*-algebra of $\QQ$, then
$\A$ is a masa of $\QQ$ if and only if for every self-adjoint
$S\in\QQ\setminus\A$ there is $R\in\A$ which does not commute with $S$.
\end{lemma}

\begin{lemma}[Lemma 2.15 of \cite{masas}]\label{masa-masaQQ} Suppose that $\A\subseteq\QQ$ is a masa of $\QQ$ which has a commutative lift.
Then $\A$ is the image under $\pi$ of a masa of $\bb$.  
\end{lemma}

\subsection{Operators in the Cohen extensions}

For a nonempty set $S$, let
$$
\PP_S=\{f\in \{0,1\}^A: A\in [S\times\N]^{<\omega}\}
$$
ordered by the reverse inclusion. Thus $\PP_S$ is the forcing adding a family of Cohen reals indexed by $S$. 
If $S$ is nonempty and countable, then $\PP_S$
 is isomorphic to the forcing adding one Cohen real. If $S\subseteq\kappa$, then
$\PP_\kappa$ is isomorphic to $\PP_S\times\PP_{\kappa\setminus S}$.
We write $\PP=\PP_{\{0\}}$ for the forcing adding one Cohen real.

We will use the fact that every name for 
a real in a $\PP_\kappa$-extension depends on only countably many coordinates. 
Consequently, if $G\subseteq\PP_\kappa$ is generic and $x\in V[G]$ is a real, 
then there is a countable $S\subseteq\kappa$ such that
$
x\in V[G\cap\PP_S].
$

We now explain how ground-model operators are interpreted in forcing extensions. Fix an orthonormal basis
$$
\E=(e_n)_{n\in\N}
$$
of $\ell_2$. If $T\in\bb$ belongs to the ground model, let
$$
t_{m,n}=\langle T(e_n),e_m\rangle
$$
for $m,n\in\N$. In every forcing extension the same matrix $(t_{m,n})_{m,n\in\N}$ 
determines a unique bounded operator on the $\ell_2$ of the extension. 
Indeed, it first determines the operator on the finite linear 
combinations of elements of $\E$ with rational complex coefficients, and the ground-model estimate
$$
\left\|T\left(\sum_{n\in F}a_ne_n\right)\right\|
\leq
\|T\|\left\|\sum_{n\in F}a_ne_n\right\|
$$
for finite $F\subseteq\N$ and rational complex numbers $a_n$ remains valid in every forcing extension. 
Since such vectors form a dense subset of the $\ell_2$ of the extension, 
this operator extends uniquely to all of $\ell_2$, with the same norm.
Clearly this extension does not depend on the choice of the orthonormal basis used in its construction. 
Whenever $T\in\bb$ belongs to the ground model, we use $\check T$ for the canonical name
 of this basis-independent extension. Thus, in expressions such as
$$
p\forces[\dot S,\check T]\notin\K,
$$
the symbol $\check T$ does not denote merely the ground-model function
 $T$ with its old domain. It denotes the unique bounded operator on the $\ell_2$ 
of the extension which extends the ground-model operator $T$.

If $P\in\bb$ is a projection in the ground model, then
$$
\PP\forces\check P^2=\check P=\check P^*,
$$
so $\PP$ forces that $\check P$ is a projection. This follows from the uniqueness 
of the continuous extensions of the operators appearing in the equalities $P^2=P=P^*$. If $K\in\K$, then
$$
\PP\forces\check K\in\K.
$$
Indeed, $K$ is a norm limit of a ground-model sequence of finite-rank operators, 
whose extensions remain finite-rank operators and converge in norm to $\check K$ in every forcing extension.

Conversely, every operator $T\in\bb$ in a forcing extension is determined by its matrix coefficients
$$
\big(\langle T(e_n),e_m\rangle\big)_{m,n\in\N}
$$
with respect to any fixed ground-model orthonormal basis, and hence is coded by an element of $2^\N$.
We will use nice names for such codes. As usual, identifying  $2^\N$ with
$\wp(\N)$, a $\PP$-name $\dot r$ for an element of $2^\N$ is called nice if
$\dot r=\{\langle\check n,p\rangle:n\in\N,\ p\in A_n\}$, where each
$A_n\subseteq\PP$ is an antichain (see \cite{kunen}). For every $\PP$-name $\dot x$
for an element of $2^\N$ there is a nice name $\dot y$ such that $\PP$
forces that $\dot x=\dot y$. Accordingly,
a $\PP$-name for an operator will be called nice if its real code is given
by a nice name.

Therefore, if $G\subseteq\PP_\kappa$ is generic and $T\in V[G]\cap\bb$, then there
is a countable $S\subseteq\kappa$ and an operator
$
T_S\in V[G\cap\PP_S]
$
whose matrix with respect to $(e_n)_{n\in\N}$ is the same as the matrix of $T$.
The operator $T$ is the unique bounded extension of $T_S$ from
$\ell_2$ in ${V[G\cap\PP_S]}$ to $\ell_2$ in ${V[G]}$. Equivalently, $T_S$ and $T$
agree on the dense subspace of finite linear combinations of the vectors $e_n$
with rational complex coefficients.

As in the case of ground-model operators considered above, we will usually
identify $T_S$ with its unique bounded extension to any further forcing extension
and omit the distinction between these operators from the notation.

By choosing representatives in $\bb$, the same argument applies to elements
of $\QQ$. More precisely, if $A\in\QQ$ belongs to a forcing extension, choose
$T\in\bb$ such that $\pi(T)=A$. Then $T$ is determined by a real and hence,
for some countable set of coordinates, is the unique extension of an operator
from the corresponding intermediate model. We will make the same
identification for elements of $\QQ$.

Every masa $\M\subseteq\bb$ is also coded by a real. For example, if
$\M=\D(\E)$ is atomic for some orthonormal basis $\E$, then it is coded by
 the  orthonormal basis $\E$. In the general
case,
if $\M$ is a masa of $\bb$, then, by the description of masas given in Section 2.5, it is
of the form
$$
\M=U^*L^\infty(\mu)U
$$
for a unitary operator
$
U:\ell_2\longrightarrow L_2(\mu),
$ and a  measure $\mu$ such that $L_2(\mu)$ is separable.
Hence $\M$ is coded by a real coding $U$. 

Consequently, if $G\subseteq\PP_\kappa$ is generic and $\M\subseteq\bb$
is a masa in $V[G]$, then there is a countable $S\subseteq\kappa$ and a real
$r\in V[G\cap\PP_S]$ such that $r$ codes a masa $\M_S$ in
$V[G\cap\PP_S]$, while the interpretation of the same code in $V[G]$
is $\M$. Every operator in $\M_S$ extends to an operator in $\M$, although
$\M$ may contain operators which do not belong to the intermediate model.
We will usually use the same notation for the interpretations of a fixed code
in different forcing extensions and omit this distinction.

\subsection{Small Boolean algebras in the Cohen model}

Since our main result, Theorem \ref{main} concerns Boolean
algebras of cardinality $\omega_1$ without countably generated ultrafilters in the Cohen model,
it is worthwhile to ask what such algebras are. Definitely there are such algebras in {\sf ZFC}, for
example  the free Boolean algebra  of cardinality $\omega_1$. The following proposition 
provides more examples. Recall that the reaping number of a
Boolean algebra $\A$ is the minimal cardinality of $\mathcal R\subseteq \A\setminus \{0\}$ such that
for every $A\in \A$ there is $R\in \mathcal R$ such that $R\leq A$ or $R\wedge A=0$ (see
Section 3 of \cite{hrusak}).

\begin{proposition}\label{reaping} Suppose that $\A$ is a Boolean algebra
whose reaping number is uncountable and $S$ is any nonempty set. Then $\PP_S$ forces that
$\check\A$ does not admit a countably generated ultrafilter.
\end{proposition}
\begin{proof}
Suppose that $\dot{\mathcal U}$ is a name for a countable basis of an ultrafilter
of $\check \A$. Since $\PP_S$ satisfies c.c.c. there is a countable $\mathcal V\subseteq\A\setminus\{0\}$ such that
$\PP_S$ forces that $\dot{\mathcal U}\subseteq \check{\mathcal V}$. 
But $\mathcal V$ is not a reaping family, so there is $A\in \A$ which
splits each element of $\mathcal V$ which is impossible because either $A$ or $-A$ will be 
above an element of the interpretation of $\dot{\mathcal U}$ in the generic
extension by $\PP_S$.
\end{proof}

So Boolean algebras like the ground model $\wp(\N)/Fin$ are examples
of algebras of cardinality $\omega_1$ with no countably generated ultrafilter in the Cohen model.
The same applies to the ground-model Lebesgue measure algebra: its reaping
number is uncountable, as follows immediately from the Baire category theorem
applied to the complete metric $d(A,B)=\mu(A\triangle B)$.
However,
in the Cohen model not having a countably generated ultrafilter
is a strong  limiting condition for Boolean algebras of cardinality $\omega_1$ as seen in the following:

\begin{proposition}\label{ind} In the model obtained from a model of 
{\sf CH} by adding $\kappa\geq\omega_2$ Cohen reals every Boolean algebra
of cardinality less than $\kappa$
which does not admit a countably generated ultrafilter  admits an uncountable
independent family.
\end{proposition}
\begin{proof}
This is essentially the result of \cite{pa}. We present the proof for  the convenience of the reader.
Let $\A$ be a Boolean algebra in the Cohen model which has cardinality less than $\kappa$
and does not admit an uncountable independent family. We will show that
$\A$ admits a countably generated ultrafilter. If $\A$ has an atom, then the principal ultrafilter generated by this atom is
countably generated. Thus we may assume that $\A$ is atomless.

An atomless Boolean subalgebra $\B$ of $\A$ is called deep (Definition 1 of \cite{pa}) if
for every $A\in \A$ and $B\in \B\setminus\{0\}$ there is $C\in \B\setminus \{0\}$
such that $C\leq A\wedge B$ or $C\leq B-A$. By Lemma 3 of \cite{pa} since $\A$ is atomless
and does not admit an uncountable independent family,  there is a countable
deep subalgebra $\B$ of $\A$ (essentially we just take a countable elementary submodel of $\A$).

Since $\A$ is of cardinality less than $\kappa$ in the model obtained from a model of 
{\sf CH} by adding $\kappa\geq\omega_2$ Cohen reals we have a Cohen real over
an intermediate submodel which contains $\B$ and $\A$. As $\B$ is countable and atomless
the forcing $\B\setminus\{0\}$ is forcing equivalent to $\PP$. So, there is a generic filter in 
$\B\setminus\{0\}$ over an intermediate submodel containing $\A$. As in the proof of Lemma 5 of \cite{pa}
note that since $\B$ is deep in $\A$ the sets
$$D_A=\{C\in \B\setminus\{0\}: C\leq A\ \hbox{or}\  C\leq -A\}$$
are dense in $\B\setminus\{0\}$ for every $A\in \A$ and they belong to the intermediate
model over which there is a Cohen real. It follows that the generic ultrafilter
in $\B\setminus\{0\}$ corresponding to the Cohen real generates an ultrafilter in $\A$.
Since $\B$ is countable, this ultrafilter is countably generated. 
\end{proof}

Note that, by the above proposition, in the Cohen model the Boolean algebras of clopen subsets of spaces like
 Malykhin's space of \cite{malykhin} (cf. \cite{juhasz}), i.e.,
with countable independence (even countable tightness) and with no countably generated ultrafilters
 must have at least cardinality continuum.
The same applies to Boolean algebras of clopen subsets of totally disconnected Efimov spaces by 
Theorem 2.1 of \cite{geschke} and a remark after it. We do not know
if such Boolean algebras of cardinality continuum
can correspond to masas of $\QQ$ in the Cohen model or even in {\sf ZFC}
(by the results of \cite{masas} they do correspond to masas of $\QQ$ under {\sf CH}). 

\section{Characters of points in Gelfand spaces of masas  of $\QQ$}

The purpose of this section is to prove the operator-algebraic ingredient
of Theorem \ref{p-main}. We use ${\sf MA}_{\sigma-\hbox{\rm centered}}$
to construct a projection witnessing that a point of small character in a compact space $K$
prevents a commutative C*-subalgebra of $\QQ$  $*$-isomorphic to $C(K)$  from being a masa.

\begin{proposition}\label{masigma} Let $\kappa$ be an infinite cardinal such that
${\sf MA}_{\sigma-\hbox{\rm centered}}(\kappa)$ holds.
If $K$ is a compact Hausdorff space such that
$C(K)$ is $*$-isomorphic to a masa in $\QQ$, then $K$ has no point of character  $\kappa$.
\end{proposition}
\begin{proof}
Suppose that $\iota:C(K)\rightarrow\A\subseteq\QQ$ is a unital $^*$-isomorphism 
and $x\in K$ has a neighbourhood basis of cardinality $\kappa$. Let $0_x$ be
 the set of all $f\in C(K)$ which are zero on some neighbourhood of $x$. 
Let $\{f_\xi:\xi<\kappa\}\subseteq0_x$ be such a collection of real-valued 
functions of norm one that $0\leq f_\xi\leq1$ and for every open $U\subseteq K$
 with $x\in U$ there is $\xi<\kappa$ such that
$f_\xi\!\upharpoonright(K\setminus U)=1$. The existence of the neighbourhood
 basis at $x$ of cardinality $\kappa$ guarantees the existence of such a collection.

For $\xi<\kappa$, by Proposition II.5.1.5 of \cite{blackadar}, let $T_\xi\in\bb$ be a positive operator such that
$$
\pi(T_\xi)=\iota(f_\xi)\quad\hbox{and}\quad\|T_\xi\|=1.
\leqno(1)
$$

We will find a noncompact projection $P\in\bb$, $\pi(P)\notin\A$, such that $\pi(P)$ 
commutes with all elements of $\iota[0_x\cup\{1\}]$.
 This is enough to conclude that $\A$ is not a masa because $\mathbb C1+0_x$
 is closed under multiplication and $^*$, includes constants and separates points of $K$,
 and so by the Weierstrass--Stone theorem $\iota[0_x\cup\{1\}]$ generates $\A$. 
Hence $\pi(P)=\pi(P)^*$ would commute with the entire 
$\A$, and so $C^*(\A\cup\{\pi(P)\})$ would be a witness for $\A$ not being a masa.

In fact, it is enough to have $P\in \bb$ such that
\begin{enumerate}
\item[(2)] $P$ is a noncompact projection,
\item[(3)] $T_\xi P$ is compact for every $\xi<\kappa$,
\item[(4)] $\pi(P)\notin\A$,
\end{enumerate}
because then, as any element $f\in0_x$ satisfies $f=ff_\xi$ for some $\xi<\kappa$, by (1) and (3) we have
$$
\iota(f)\pi(P)=\iota(f)\iota(f_\xi)\pi(P)=\iota(f)\pi(T_\xi P)=0.
\leqno(5)
$$
Taking adjoints in (5), we obtain $\pi(P)\iota(f)=0$. Thus $\pi(P)$ commutes
 with all elements of $\iota[0_x\cup\{1\}]$. 
The remainder of the proof is devoted to obtaining $P$ as in (2)--(4).

Fix an orthonormal basis $\E=(e_n)_{n\in\N}$ for $\ell_2$. 
When we mention supports of vectors in $\ell_2$ or coordinates 
of vectors in $\ell_2$, this is meant with respect to $\E$. 
To use ${\sf MA}_{\sigma\text{-centered}}(\kappa)$, 
we define a partial order $\SSS$ to consist of triples
$$
p=((v_i^p:i<n_p), n_p, F_p),
$$
where
\begin{enumerate}
\item[(a)] $n_p\in\N$,
\item[(b)] the $v_i^p$'s, for $i<n_p$, are finitely
 supported unit vectors with rational coordinates from
 $\ell_2$ which have pairwise disjoint supports,
\item[(c)] $F_p$ is a finite subset of $\kappa$.
\end{enumerate}

We declare $p\leq q$ if and only if
\begin{enumerate}
\item[(d)] $n_p\geq n_q$ and $F_p\supseteq F_q$,
\item[(e)] $v_i^p=v_i^q$ for all $i<n_q$,
\item[(f)] $\|T_\xi(v_i^p)\|\leq2^{-i}$ for all $i\in n_p\setminus n_q$ and all $\xi\in F_q$.
\end{enumerate}

It is clear that this is a partial order which is $\sigma$-centered 
since elements $p\in\SSS$ with fixed first two coordinates are all compatible 
and there are only countably many possibilities for the values of the first two coordinates. 
Now we will show that certain subsets of $\SSS$ are dense in $\SSS$.

First, for $k\in\N$, consider the following subsets of $\SSS$:
$$
\mathbb E_k=\{p\in\SSS:n_p>k\}.
$$
We prove that each $\mathbb E_k$ is dense in $\SSS$. 
To do so, given $q\in\SSS$, we will find $p\leq q$ such that $n_p=n_q+1$. 
Applying this procedure finitely many times, we will be able to construct a condition
 stronger than $q$ which is in $\mathbb E_k$, as required for the density of $\mathbb E_k$.

Let $m\in\N$ be such that the supports of all $v_i^q$, for $i<n_q$,
 are contained in $\{0,\ldots,m\}$. Consider the projection $Q\in\bb$ 
onto the space spanned by $\{e_n:n\leq m\}$. Note that, as each $f_\xi$,
 for $\xi\in F_q$, is in $0_x$, there is a real-valued $f\in C(K)$ of norm one such that
$$
f_\xi f=0\quad\hbox{for every }\xi\in F_q.
\leqno(6)
$$
Indeed, the finitely many functions $f_\xi$, for $\xi\in F_q$, 
are zero on a common neighbourhood of $x$, and $f$ can be chosen with support contained in that neighbourhood.

Let $S\in\bb$ be self-adjoint such that $\pi(S)=\iota(f)$. By the Weyl--von Neumann 
theorem there is an orthonormal basis $\E'=(e_n')_{n\in\N}$ such that $S=S'+R$, 
where $S'\in\D(\E')$ and $R\in\K$. Let $(\alpha_n)_{n\in\N}$ 
be such that $S'(e_n')=\alpha_ne_n'$ for each $n\in\N$.

As
$$
\|\pi(S')\|=\|\pi(S)\|=\|\iota(f)\|=\|f\|=1,
$$
there are $\varepsilon>0$ and an infinite $X\subseteq\N$ such that $|\alpha_n|>\varepsilon$ 
for all $n\in X$. By (6), for every $\xi\in F_q$ we have
$$
\pi(T_\xi S')=\pi(T_\xi)\pi(S')=\iota(f_\xi)\iota(f)=\iota(f_\xi f)=0,
$$
and hence $T_\xi S'\in\K$. Therefore
$$
\varepsilon\lim_{n\in X}\|T_\xi(e_n')\|\leq\lim_{n\in X}\|\alpha_nT_\xi(e_n')\|=\lim_{n\in X}\|T_\xi S'(e_n')\|=0.
\leqno(7)
$$
for every $\xi\in F_q$. Since $Q$ is a finite-rank projection, we also have
$$
\lim_{n\in X}\|Q(e_n')\|=0.
\leqno(8)
$$

It follows from (7) and (8) that we can take $j\in X$ sufficiently large such that
$$
\|T_\xi(e_j')\|<\frac{1}{10\cdot2^{n_q}}\quad\hbox{for every }\xi\in F_q
$$
and
$$
\|Q(e_j')\|<\frac{1}{10\cdot2^{n_q}}.
$$
Let $v'=e_j'-Q(e_j')$. Then the support of $v'$ is disjoint from the supports of $v_i^q$, for $i<n_q$, and
$$
1-\frac{1}{10\cdot2^{n_q}}\leq\|v'\|\leq1+\frac{1}{10\cdot2^{n_q}}.
$$
Moreover, by (1), for every $\xi\in F_q$ we have
$$
\|T_\xi(v')\|\leq\|T_\xi(e_j')\|+\|T_\xi Q(e_j')\|\leq\|T_\xi(e_j')\|+\|Q(e_j')\|<\frac{1}{5\cdot2^{n_q}}.
$$

Put $w=v'/\|v'\|$. Since $10\cdot2^{n_q}\geq10$, it follows that, for every $\xi\in F_q$,
$$
\|T_\xi(w)\|=\frac{\|T_\xi(v')\|}{\|v'\|}<\frac{1}{5\cdot2^{n_q}}
\left(1-\frac{1}{10\cdot2^{n_q}}\right)^{-1}\leq \frac{1}{5\cdot2^{n_q}}\cdot{10\over 9}
\leq\frac{2}{9\cdot2^{n_q}}.
\leqno(9)
$$

The finitely supported unit vectors with rational coordinates whose supports are contained
 in $\{m+1,m+2,\ldots\}$ are dense in the unit sphere of the closed subspace spanned
 by $\{e_n:n>m\}$. Therefore we can choose such a vector $v$ satisfying
$$
\|v-w\|<\frac{1}{4\cdot2^{n_q}}.
$$
Its support is disjoint from the supports of $v_i^q$, for $i<n_q$. 
Moreover, by (1) and (9), for every $\xi\in F_q$,
$$
\|T_\xi(v)\|\leq\|T_\xi(w)\|+\|T_\xi(v-w)\|\leq\|T_\xi(w)\|+\|v-w\|
<\left(\frac{2}{9}+\frac14\right)\frac1{2^{n_q}}<\frac1{2^{n_q}}.
\leqno(10)
$$

So put $v_i^p=v_i^q$ for $i<n_q$ and $v_{n_q}^p=v$. Then
$$
p=((v_i^p:i<n_q+1),n_q+1,F_q)\in\SSS,
$$
and $p\leq q$ by (10). Applying this construction finitely
many times proves that every $\mathbb E_k$ is dense in $\SSS$.

Now, for $\xi<\kappa$, consider the following subsets of $\SSS$:
$$
\mathbb D_\xi=\{p\in\SSS:\xi\in F_p\}.
$$
For every $\xi<\kappa$, the set $\mathbb D_\xi$ is dense in $\SSS$, since
$$
((v_i^q:i<n_q) ,n_q,F_q\cup\{\xi\})\leq((v_i^q:i<n_q),n_q,F_q).
$$

So we can use ${\sf MA}_{\sigma\text{-centered}}(\kappa)$ to obtain a filter $\mathbb G\subseteq\SSS$ such that
$$
\mathbb G\cap\mathbb D_\xi\neq\varnothing\quad\hbox{and}\quad\mathbb G\cap\mathbb E_k\neq\varnothing
$$
for every $\xi<\kappa$ and $k\in\N$. Note that the compatibility of all
 elements in $\mathbb G$ implies that $v_i^p=v_i^q$ for any 
$p,q\in\mathbb G$ such that $i<n_p,n_q$. So let $v_i=v_i^p$ for any $p\in\mathbb G$ 
such that $n_p>i$, which exists since $\mathbb G\cap\mathbb E_i\neq\varnothing$.

Also, given $p\in\mathbb G$ with $\xi\in F_p$ and any $i\geq n_p$, 
we find $q\in\mathbb G\cap\mathbb E_i$ and further 
$r\in\mathbb G$ such that $r\leq q,p$. So, by (f),
$$
\|T_\xi(v_i)\|\leq\frac1{2^i}.
\leqno(11)
$$
It follows from (11) that for every $\xi<\kappa$ we have
$$
\sum_{i\in\N}\|T_\xi(v_i)\|\leq\sum_{i<n_p}\|T_\xi(v_i)\|+\sum_{i\geq n_p}\frac1{2^i}<\infty.
\leqno(12)
$$

So if $P_0, P_1$ are the orthogonal projections onto the closed subspaces
 of $\ell_2$ spanned by $\{v_{2i}:i\in\N\}$ and $\{v_{2i+1}:i\in\N\}$, 
respectively, then $P_0$ and $P_1$ are noncompact projections, 
because their ranges are infinite-dimensional. Thus both satisfy (2).

Moreover, $T_\xi P_j$ is compact for $j=0, 1$. Indeed, let $P_{j, n}$ be 
the projection onto the space spanned by the first $n$ vectors used
 in the definition of $P_j$. Then $T_\xi P_{j, n}$ has finite rank and
$$
\|T_\xi(P_j-P_{j, n})\|\leq\sum_{i\geq n}\|T_\xi(v_{2i+j})\|,
$$
which converges to $0$ by (12).
Thus $T_\xi P_j$ is a norm limit of finite-rank operators and hence is compact. 
So we obtained (2) and (3) for both $P_0$ and $P_1$.

To see (4) for one of them, suppose that $\pi(P_j)\in\A$, and put $g_j=\iota^{-1}(\pi(P_j))$. By (1) and (3),
$$
\iota(f_\xi g_j)=\iota(f_\xi)\pi(P_j)=\pi(T_\xi P_j)=0
$$
for every $\xi<\kappa$. Thus $f_\xi g_j=0$ for every $\xi<\kappa$.
 Given any neighbourhood $U$ of $x$, take $\xi<\kappa$ such that 
$f_\xi=1$ on $K\setminus U$. It follows that $g_j=0$ on $K\setminus U$, 
and hence $g_j$ can be nonzero only at $x$. 
Since $P_j$ is noncompact, $\pi(P_j)\neq0$, so $g_j\neq0$. 
As $g_j$ is a projection in $C(K)$, this implies that $x$ is isolated in $K$ and
$$
g_j=1_{\{x\}}.
$$
Therefore both $\pi(P_0)$ and $\pi(P_1)$ cannot belong to $\A$ because they are distinct
as orthogonal and nonzero projections,
so one of them satisfies (4), which completes the proof.
\end{proof}

%For $n\in \N$ we define $\HH_n = \{v\in \HH: \|v\|=1\ \& \ v(m) = 0\ \hbox{for all}\  m < n\}$.

\section{Cohen indestructible masas of $\QQ$}

This section contains the proof of Theorem \ref{main}. The first two lemmas
provide a local modification of a Boolean embedding which destroys a prescribed
Cohen name for a commuting operator while preserving the required Boolean
isomorphism type. We then iterate this procedure under {\sf CH} and finally
transfer the construction to the full Cohen extension.

Lemmas \ref{substitution} and \ref{main-lemma} are forcing versions
of Lemmas 3.2 and 3.3 of \cite{masas}. The crucial difference between Lemma
\ref{substitution} and Lemma 3.2 of \cite{masas} is that here
we only have a forcing name $\dot T$ for an operator rather than an actual operator $T$ as in \cite{masas}.
This makes the construction of the ground model projection $Q$ more involved as it needs to
predict what different conditions may force about $\dot T$. Such a prediction is possible
because we assume the existence of a ground model projection $P$ which has some information
about $\dot T$, namely we assume that $P$ splits a possibly new ultrafilter which is C*-split by $\dot T$.
When Lemma \ref{substitution} is used through Lemma \ref{main-lemma}
in the proof of Proposition \ref{ch-proposition}, this is provided by the assumption
that $\PP$ does not add a countably generated ultrafilter to $\A$. In the final
application this follows from the hypothesis of Theorem \ref{main} by the
factorization argument in Proposition \ref{cohen-proposition} which is 
possible since the algebra is of small cardinality.

\begin{lemma}\label{substitution} Let $\E$ be an orthonormal basis of $\ell_2$.
Suppose that
\begin{enumerate}
\item[(1)] $\A\subseteq\B\subseteq\D(\E)$ are countable Boolean subalgebras of projections,
\item[(2)] $\B$ is a simple extension of $\A$ by $P$, where $P\in\D(\E)$ is a projection,
\item[(3)] $\pi|\B$ is a Boolean isomorphism,
\item[(4)] $\J=\J_{\A,P}$ is as in Definition \ref{def-ideal-simple},
\item[(5)] $\dot T, \dot c_1, \dot c_2, \dot\U$ are $\PP$-names and $p\in\PP$ is such that
$p$ forces that
\begin{enumerate}
\item[(5A)] $\dot T\in\bb$ is self-adjoint,
\item[(5B)] $\dot\U$ is an ultrafilter of $\check\A$ and
$\check\J\cap\dot\U=\emptyset$,
\item[(5C)] $\dot c_1,\dot c_2\in\R$ and $\dot c_1<\dot c_2$,
\item[(5D)] for every $R\in\dot\U$ and every projection $S\in\bb$ such that
$R-S\in\K$,
$$\dot c_1,\dot c_2\in
\sigma_{S\bb S}
(S\dot T S)
$$
where the spectrum is computed in the corner $S\bb S$, whose unit is $S$.
\end{enumerate}
\end{enumerate}
Then there is an ideal $\J\subseteq\I\subseteq\A$, a projection $Q\in\bb$,
a Boolean monomorphism $h:\pi[\B]\rightarrow\QQ$ and $p'\leq p$ such that
\begin{enumerate}
\item[(6)] $p'\forces[\dot T,\check Q]\notin\K$,
\item[(7)] $RQ,QR,R(I-Q),(I-Q)R\notin\K$ for all $R\in\A\setminus\I$,
\item[(8)] $(P-Q)R,R(P-Q)\in\K$ for all $R\in\I$,
\item[(9)] $[Q,R]\in\K$ for every $R\in\A$,
\item[(10)] $h$ is the identity on $\pi[\A]$ and $h(\pi(P))=\pi(Q)$.
\end{enumerate}
\end{lemma}

\begin{proof}
Let $\E=(e_k)_{k\in\N}$. For $A\subseteq\N$ by $P_A$ we will mean the projection
in $\D(\E)$ such that $P_A(e_k)=e_k$ if $k\in A$ and $P_A(e_k)=0$ if
$k\notin A$. All projections in $\D(\E)$ are of this form for some
$A\subseteq\N$, in particular $P=P_X$ for some $X\subseteq\N$.

By strengthening $p$, choose $p'\leq p$ and
a positive rational number $\varepsilon$ such that
$$
p'\forces \dot c_1+\check \varepsilon<\dot c_2.
$$
In the remainder of the proof we work below $p'$.

Let $(A_n)_{n\in\N}$ be an enumeration of all subsets of $\N$ such that
$\A=\{P_{A_n}:n\in\N\}$ with $A_0=\N$. Let $(B_n)_{n\in\N}$ be an
enumeration, possibly with repetitions, of all subsets of $\N$ such that
$\J=\{P_{B_n}:n\in\N\}$ with $B_0=\emptyset$.

Let $(p_n:n\in\N)$ be an enumeration of all conditions in $\PP$ stronger than
$p'$, where each such condition is repeated infinitely many times. By induction
on $n\in\N$ we will construct $(F_n)_{n\in\N}$, $(Q_n)_{n\in\N}$ and
$q_n\leq p_n$ such that for every $n,n'\in\N$ the following hold:
\begin{enumerate}
\item[(a)] $F_n$ is a finite subset of $\N$ disjoint from
$\bigcup_{l\leq n}B_l$,
\item[(b)] $F_n\cap F_{n'}=\emptyset$ whenever $n\neq n'$,
\item[(c)] $F_n$ is a subset of an atom of the Boolean algebra generated by
$\{A_l:l\leq n\}$,
\item[(d)] $Q_n\in\bb$ is a projection such that
$\langle Q_n(e_k),e_{k'}\rangle\neq0$ implies $k,k'\in F_n$ for all
$k,k'\in\N$,
\item[(e)] $q_n\forces
\|[P_{\check F_n}\dot TP_{\check F_n},\check Q_n]\|>\check\varepsilon/3$.
\end{enumerate}

Once we have such objects we put
$$
B=\N\setminus\bigcup_{n\in\N}F_n,
$$
and we define
$$
Q(e_k)=
\begin{cases}
P(e_k)&\text{if $k\in B$},\\
Q_n(e_k)&\text{if $k\in F_n,\ n\in\N$}.
\end{cases}
$$

First let us present an inductive construction satisfying (a)--(e). Suppose that
we are done before stage $n\in\N$. Let
$\{C_1'',\dots,C_{k_n}''\}$ be all atoms of the subalgebra of $\wp(\N)$
generated by $\{A_l:l\leq n\}$, where $k_n\in\N$. Let $s_n\leq p_n$ and
$1\leq i\leq k_n$ be such that
$$
s_n\forces P_{\check C_i''}\in\dot\U.
$$
Such $s_n$ and $i$ exist because $\{C_1'',\dots,C_{k_n}''\}$ forms a
partition of unity in $\A$. Put
$$
C_i'=C_i''\setminus\bigcup_{l\leq n}B_l
$$
and
$$
C_i=C_i'\setminus\bigcup_{l<n}F_l.
$$
Since the join of the projections $P_{B_l}$, for $l\leq n$, belongs to $\J$,
condition (5B) implies that
$$
s_n\forces P_{\check C_i'}\in\dot\U.
$$

The set $C_i'$ is infinite. Indeed, $s_n$ forces that $P_{\check C_i'}$
belongs to an ultrafilter, and hence $\pi(P_{C_i'})\neq0$. If $C_i'$ were
finite, then $\pi(P_{C_i'})=0$, contrary to (3). Since only finitely many
finite sets $F_l$, for $l<n$, have been removed from $C_i'$, the set $C_i$
is infinite as well.

The operator $P_{C_i'}-P_{C_i}$ has finite rank. Hence, applying (5D) with
$$
R=P_{C_i'},\qquad S=P_{C_i},
$$
we obtain
$$
s_n\forces\dot c_1,\dot c_2\in
\sigma_{P_{\check C_i}\bb P_{\check C_i}} (P_{\check C_i}\dot TP_{\check C_i}).
$$
where the spectrum is computed in the corner $P_{\check C_i}\bb P_{\check C_i}$
whose unit is $P_{\check C_i}$.
In the forcing extension apply Lemma \ref{commutator} to the operator
$P_{C_i}\dot TP_{C_i}$ on the Hilbert space generated by
$\{e_k:k\in C_i\}$.  Since $s_n$ forces that
$\dot c_2-\dot c_1>\check\varepsilon$, it forces that there are a finite
$F\subseteq C_i$ and a projection $S$ supported by $F$ such that
$$
\|[P_F\dot TP_F,S]\|>\check\varepsilon/3.
$$
By strengthening $s_n$, we may decide the finite set $F$. Thus there are
$r_n\leq s_n$, a finite set $F_n\subseteq C_i$ and a $\PP$-name
$\dot S_n$ such that $r_n$ forces that $\dot S_n$ is a projection supported
by $\check F_n$ and
$$
r_n\forces
\|[P_{\check F_n}\dot TP_{\check F_n},\dot S_n]\|>\check\varepsilon/3.
$$

The ground-model projections on the finite-dimensional space $\ell_2(F_n)$
are norm dense in the set of all projections on this space in every forcing
extension. Indeed, an orthonormal basis of the range of a projection can be
approximated by vectors with rational complex coordinates and then
orthonormalized. Consequently, there are a ground-model projection
$Q_n\in\bb$ supported by $F_n$ and $q_n\leq r_n$ such that
$$
q_n\forces
\|[P_{\check F_n}\dot TP_{\check F_n},\check Q_n]\|>\check\varepsilon/3.
$$
Thus (a)--(e) are satisfied, and the inductive construction is complete.

Let us define
$$
\I=\{R\in\A:p'\forces\check R\notin\dot\U\}.
$$
The set $\I$ is a proper ideal of $\A$. Indeed, this follows immediately from
the fact that $p'$ forces that $\dot\U$ is an ultrafilter. Moreover,
$\J\subseteq\I$ by (5B).

Finally, before going to the proof of (6)--(10), note that for every $R\in\A$
we have the following:
\begin{enumerate}
\item[(f)] $QP_{F_n}=P_{F_n}QP_{F_n}=Q_n$ for all $n\in\N$,
\item[(g)] $Q_n=P_{F_n}Q$ for all $n\in\N$,
\item[(h)] if $R\notin\I$, then $RP_{F_n}=P_{F_n}R=P_{F_n}$ for infinitely
many $n\in\N$, and for all sufficiently large $n\in\N$ either
$RP_{F_n}=P_{F_n}$ or $RP_{F_n}=0$,
\item[(i)] if $R\in\I$, then $RP_{F_n}=P_{F_n}R=0$ for all but finitely many
$n\in\N$.
\end{enumerate}

Item (f) follows from (d) and from the definition of $Q$. Since $Q$,
$P_{F_n}$ and $Q_n$ are self-adjoint, taking adjoints in (f) gives (g).

For (h) and (i), fix $A\subseteq\N$ such that $R=P_A$. For sufficiently large
$n\in\N$, the element $A$ belongs to the algebra generated by
$\{A_l:l\leq n\}$. Since $F_n$ is contained in an atom of this algebra,
for all sufficiently large $n\in\N$ either $F_n\subseteq A$ or
$F_n\cap A=\emptyset$.

If $R\notin\I$, there is $r\leq p'$ such that
$$
r\forces\check R\in\dot\U.
$$
The condition $r$ occurs infinitely many times in the enumeration
$(p_n:n\in\N)$. At every sufficiently large stage $n$ such that $p_n=r$,
the atom $C_i''$ which is forced by $s_n\leq p_n$ to belong to $\dot\U$
must be below $A$. Hence $F_n\subseteq A$, proving the first part of (h).
The second part follows from the preceding observation.

If $R\in\I$, then
$$
p'\forces I-\check R\in\dot\U.
$$
For all sufficiently large $n\in\N$, the atom selected at stage $n$ is
therefore disjoint from $A$, and so $F_n\cap A=\emptyset$. This proves (i).

Now let us prove that $Q$ satisfies (6)--(10). We will often be using the
elementary fact that $UV\in\K$ if and only if $VU\in\K$ for self-adjoint
$U,V\in\bb$, which follows from the Schauder theorem and the fact that
$VU=V^*U^*=(UV)^*$. Thus only half of (7) and (8) needs to be proved.

Suppose that (6) fails. Then there is $r\leq p'$ such that
$$
r\forces[\dot T,\check Q]\in\K.
$$
Since the subspaces $\ell_2(F_n)$ are pairwise orthogonal, $r$ forces that
$$
\lim_{n\rightarrow\infty}\|P_{\check F_n}[\dot T,\check Q]P_{\check F_n}\|=0.
$$
Consequently, there are $r'\leq r$ and $k\in\N$ such that
$$
r'\forces
\|P_{\check F_n}[\dot T,\check Q]P_{\check F_n}\|\leq\check\varepsilon/3
$$
for every $n>k$. The condition $r'$ occurs infinitely many times in the
enumeration $(p_n:n\in\N)$. Choose $n>k$ such that $p_n=r'$. By (f) and (g),
the entire forcing $\PP$ forces that
$$
P_{\check F_n}[\dot T,\check Q]P_{\check F_n}
=[P_{\check F_n}\dot TP_{\check F_n},\check Q_n].
$$
Since $q_n\leq p_n=r'$, this contradicts (e), and so (6) holds.

For (7), let $R\in\A\setminus\I$. By (e), $Q_n$ is neither $0$ nor
$P_{F_n}$, and so
$$
\|Q_n\|=\|P_{F_n}-Q_n\|=1.
$$
For every $n\in\N$, choose $F_n$-supported unit vectors $w_n,u_n$ such that
$$
\|Q_n(w_n)\|=1\quad\hbox{and}\quad
\|(P_{F_n}-Q_n)(u_n)\|=1.
$$
By (h), there are infinitely many $n\in\N$ such that
$RP_{F_n}=P_{F_n}$. For these $n$ by (f) and  (g) we have
$$
\|RQ(w_n)\|=\|RQP_{F_n}(w_n)\|=\|Q_n(w_n)\|=1
$$
and
$$
\|R(I-Q)(u_n)\|=\|R(P_{F_n}-Q_n)(u_n)\|=1.
$$
The $w_n$s and, respectively, the $u_n$s are pairwise orthogonal. Therefore
neither $RQ$ nor $R(I-Q)$ is compact. Taking adjoints gives the remaining
two assertions in (7).

To prove (8), let $R\in\I$. By (i), there are only finitely many $n\in\N$
such that $RP_{F_n}\neq0$. By the definition of $Q$, and since
$R,P\in\D(\E)$, it follows that $QR$ is a finite dimensional perturbation
of $PR$. Hence
$$
(Q-P)R\in\K.
$$
Taking adjoints gives $R(Q-P)\in\K$, which proves (8).

For (9), let $R\in\A$ and let $A\subseteq\N$ satisfy $P_A=R$. It follows
from (h) and (i) that either $F_n\subseteq A$ or $F_n\cap A=\emptyset$ for
all $n>k$ and some $k\in\N$. Let
$$
B'=\N\setminus\bigcup_{n>k}F_n.
$$
Define $Q'$ by
$$
Q'(e_j)=
\begin{cases}
P(e_j)&\text{if $j\in B'$},\\
Q_n(e_j)&\text{if $j\in F_n,\ n>k$}.
\end{cases}
$$
It follows that $Q-Q'$ is a finite-rank operator. Moreover, the spaces spanned
by $\{e_j:j\in A\}$ and $\{e_j:j\in\N\setminus A\}$ are invariant for
$Q'$, since $P\in\D(\E)$ and either $F_n\subseteq A$ or
$F_n\cap A=\emptyset$ for all $n>k$. Thus $RQ'=Q'R$, and so (9) follows.

Note that by (9) the set $\pi[\A]\cup\{\pi(Q)\}$ generates a Boolean algebra
$\CC$ of projections of $\QQ$. We will use Lemma \ref{sikorski} to conclude
the existence of $h$ as in (10), so we need to show that
$$
\J_{\pi[\A],\pi(Q)}=\pi[\J],
$$
since $\J_{\pi[\A],\pi(P)}=\pi[\J]$ by (3) and (4).

If $R\in\J$, then $R\in\I$ and so, by (8),
$$
\pi(R)\pi(Q)=\pi(R)\pi(P)\in\pi[\A].
$$
Consequently,
$$
\pi[\J]\subseteq\J_{\pi[\A],\pi(Q)}.
$$

If $R\notin\J$, we consider two cases, $R\in\I\setminus\J$ and
$R\notin\I$, and in both of them we will prove that
$\pi(R)\notin\J_{\pi[\A],\pi(Q)}$.

If $R\in\I\setminus\J$, by (8) we have
$$
\pi(R)\pi(Q)=\pi(R)\pi(P).
$$
Since $R\notin\J_{\A,P}=\J$, the latter element does not belong to
$\pi[\A]$. Hence $\pi(R)\notin\J_{\pi[\A],\pi(Q)}$.

If $R\notin\I$, the filter generated by $R$ is disjoint from the proper ideal
$\I$. Hence there is an ultrafilter $\mathcal V$ of $\A$ which contains $R$
and is disjoint from $\I$. By (7), the ultrafilter $\pi[\mathcal V]$ of
$\pi[\A]$ is split by $\pi(Q)$, since every element of $\mathcal V$ belongs
to $\A\setminus\I$. Lemma \ref{ideal-split} therefore implies that
$$
\pi[\mathcal V]\cap\J_{\pi[\A],\pi(Q)}=\emptyset.
$$
In particular,
$$
\pi(R)\notin\J_{\pi[\A],\pi(Q)}.
$$
This proves that
$$
\J_{\pi[\A],\pi(Q)}=\pi[\J].
$$
In particular, the ideal $\J_{\pi[\A],\pi(Q)}$ is a proper ideal of $\pi[\A]$, and hence $\pi(Q)\notin\pi[\A]$.
Therefore the Boolean algebra
$\CC$ (defined just after the proof of (9)) is a simple extension of $\pi[\A]$ by $\pi(Q)$.

Finally, by (8),
$$
\pi(R)\pi(P)=\pi(R)\pi(Q)
$$
for every $R\in\J$. Thus the hypotheses of Lemma \ref{sikorski} are
satisfied for the simple extensions $\pi[\B]$ and $\CC$ of $\pi[\A]$.
Consequently, there is a Boolean isomorphism
$$
h:\pi[\B]\longrightarrow\CC
$$
which is the identity on $\pi[\A]$ and satisfies $h(\pi(P))=\pi(Q)$.
This proves (10) and completes the proof of the lemma.
\end{proof}

\begin{lemma}\label{main-lemma} Suppose that
\begin{enumerate}
\item $\A\subseteq\B\subseteq\QQ$ are countable Boolean algebras of projections,
\item $\B$ is generated over $\A$ by a projection $R\in\B\setminus\A$,
\item $p\in\PP$ forces that $\dot T\in\bb$ is self-adjoint, $\pi(\dot T)$
commutes with all elements of $\check\A$ and
$$
\pi(\dot T)\notin C^*(\check\A),
$$
\item $\dot\U$ is a $\PP$-name and $p$ forces that $\dot\U$ is an
ultrafilter of $\check\A$, $\check R$ splits $\dot\U$ and $\dot T$
C*-splits $\dot\U$.
\end{enumerate}
Then there is a projection $P\in\QQ$, a Boolean monomorphism
$g:\B\longrightarrow\QQ$ and $p'\leq p$ such that
\begin{enumerate}
\item[(5)] $[P,A]=0$ for every $A\in\A$,
\item[(6)] $g|\A$ is the identity on $\A$ and $g(R)=P$,
\item[(7)] $p'$ forces that $\check P$ does not commute with $\pi(\dot T)$.
\end{enumerate}
\end{lemma}

\begin{proof}
By Lemma \ref{calkin-lift} there are an orthonormal basis $\E$ of $\ell_2$
and a Boolean algebra $\B'\subseteq\D(\E)$ of projections such that
$$
\pi|\B':\B'\longrightarrow\B
$$
is a Boolean isomorphism. Put
$$
\A'=\pi^{-1}[\A]\cap\B'
$$
and let $R'\in\B'$ be the unique projection such that $\pi(R')=R$.
Since $\B$ is generated over $\A$ by $R$, the Boolean algebra $\B'$ is
generated over $\A'$ by $R'$. In particular, $\B'$ is a simple extension
of $\A'$ by $R'$.

Let $\dot\U'$ be a $\PP$-name such that
$$
p\forces\dot\U'=\pi^{-1}[\dot\U]\cap\check\A'.
$$
Since $\pi|\A'$ is a Boolean isomorphism onto $\A$, the condition $p$
forces that $\dot\U'$ is an ultrafilter of $\check\A'$ and
$\pi[\dot\U']=\dot\U$.

Let $\J=\J_{\A', R'}$. Since $p$ forces that $\check R$ splits $\dot\U$,
it also forces that $\check R'$ splits $\dot\U'$. Hence, by Lemma
\ref{ideal-split},
$$
p\forces\check\J\cap\dot\U'=\emptyset.
$$
Thus items (1)--(4), (5A) and (5B) of Lemma \ref{substitution} are
satisfied by $\A'$, $\B'$, $R'$, $\J$, $\dot T$ and $\dot\U'$.

By Lemma \ref{c-split} and the maximum principle, there are
$\PP$-names $\dot c_1,\dot c_2$  such that
$$
p\forces\dot c_1,\dot c_2\in\R\quad\hbox{and}\quad
\dot c_1<\dot c_2
$$
and
$$
p\forces\dot c_1,\dot c_2\in\sigma_{A\QQ A}(A\pi(\dot T)A)\ 
\hbox{for every}\  A\in\dot\U.\leqno (*)$$
We claim that item (5D) of Lemma \ref{substitution} is satisfied as well.
Suppose that $p$ forces that $\dot A'\in\dot\U'$ and $\dot S\in\bb$ is a
projection such that $\dot A'-\dot S\in\K$. Then $p$ forces that
$\pi(\dot S)=\pi(\dot A')\in\dot\U$. Hence by (*) applied to
$A=\pi(\dot S)$ the condition $p$ forces that
$$
\dot c_1,\dot c_2\in
\sigma_{\pi(\dot S)\QQ\pi(\dot S)}
(\pi(\dot S)\pi(\dot T)\pi(\dot S)).
$$
The corner $\pi(\dot S)\QQ\pi(\dot S)$ is a quotient of
$\dot S\bb\dot S$. Since the spectrum cannot grow when we pass to a quotient
(Lemma 2.5.7 of \cite{ilijas-book}), $p$ forces that
$$
\dot c_1,\dot c_2\in\sigma_{\dot S\bb\dot S}(\dot S\dot T\dot S).
$$
Thus all the hypotheses of Lemma \ref{substitution} are satisfied.

Apply Lemma \ref{substitution}. We obtain a projection $Q\in\bb$, a Boolean
monomorphism
$$
h:\pi[\B']\longrightarrow\QQ
$$
and $p'\leq p$ such that
$$
p'\forces[\dot T,\check Q]\notin\K,
$$
$[Q,A']\in\K$ for every $A'\in\A'$, and $h$ is the identity on
$\pi[\A']$ and satisfies $h(\pi(R'))=\pi(Q)$.

Put
$$
P=\pi(Q),\qquad g=h.
$$
Since $\pi[\B']=\B$ and $\pi[\A']=\A$, the map $g$ has domain $\B$,
$g|\A$ is the identity on $\A$ and
$$
g(R)=g(\pi(R'))=\pi(Q)=P.
$$
Moreover, $[Q,A']\in\K$ for every $A'\in\A'$ implies that $P$ commutes
with every element of $\A$. Finally,
$$
p'\forces[\pi(\dot T),\check P]
=\pi([\dot T,\check Q])\neq0.
$$
Thus (5)--(7) hold.
\end{proof}

We now use Lemma \ref{main-lemma} in an $\omega_1$-stage bookkeeping
construction to obtain Cohen indestructible masas under {\sf CH}.

\begin{proposition}[{\sf CH}]\label{ch-proposition}
Suppose that $\A$ is a Boolean algebra of cardinality $\omega_1=2^\omega$ such that
$\PP$ forces that $\check\A$ does not admit a countably generated ultrafilter.
Then there is a collection $\{\X_\xi:\xi<\omega_2\}$ of Boolean subalgebras
of projections in $\QQ$ such that
\begin{enumerate}
\item[(a)] $\PP$ forces that for every $\xi<\omega_2$ the C*-algebra
$C^*(\check\X_\xi)$ is a masa of $\QQ$ without a commutative lift,
\item[(b)] for every $\xi<\omega_2$ the Boolean algebra $\X_\xi$ is
isomorphic to $\A$,
\item[(c)] $\PP$ forces that $C^*(\check\X_\xi)$ and
$C^*(\check\X_\eta)$ are unitarily non-equivalent whenever
$\xi<\eta<\omega_2$.
\end{enumerate}
\end{proposition}

\begin{proof}
First we will show how, for a family $\FF$  of cardinality $\omega_1$ of $\PP$-names  
for masas of $\QQ$, one can construct a Boolean algebra
$\B$ of projections in $\QQ$ such that
\begin{enumerate}
\item[(i)] $\B$ is isomorphic to $\A$,
\item[(ii)] $\PP$ forces that $C^*(\check\B)$ is a masa of $\QQ$,
\item[(iii)] $\PP\forces C^*(\check\B)\neq\dot\CC$
for every $\dot\CC\in\FF$.
\end{enumerate}

 Let
\begin{itemize}
\item $X,Y,Z$ be a partition of $\omega_1$ into three uncountable sets,
\item $\{A_\alpha:\alpha\in X\}$ be an enumeration of all elements of $\A$,
\item $\{(p_\alpha,\dot T_\alpha):\alpha\in Y\}$ be an enumeration of
all pairs such that $p_\alpha\in\PP$ and $\dot T_\alpha$ is a nice
$\PP$-name for a self-adjoint element of $\bb$,
\item $\{(p_\alpha,\dot\CC_\alpha):\alpha\in Z\}$ be an enumeration of
the Cartesian product of $\PP$ and $\FF$.
\end{itemize}
The existence of these enumerations follows from {\sf CH} and the
countability of $\PP$.

By induction on $\alpha<\omega_1$ we construct elements
$A_\alpha'\in\A$, projections $B_\alpha\in\QQ$, countable Boolean
algebras $\A_\alpha\subseteq\A$ and $\B_\alpha\subseteq\QQ$, and
Boolean isomorphisms $h_\alpha$ such that
\begin{enumerate}
\item $\A_\alpha$ is the Boolean subalgebra of $\A$ generated by
$\{A_\beta':\beta<\alpha\}$.
\item The projections $\{B_\beta:\beta<\alpha\}$ commute,
$\B_\alpha$ is the Boolean algebra generated by them, and
$$
h_\alpha:\A_\alpha\longrightarrow\B_\alpha
$$
is a Boolean isomorphism satisfying $h_\alpha(A_\beta')=B_\beta$
for every $\beta<\alpha$.
\item If $\alpha\in X$ and
$
\{\gamma\in\alpha\cap X:A_\gamma\notin\A_\alpha\}
$
is nonempty, then
$
A_\alpha'=A_\beta$, and 
$$\beta=\min\{\gamma\in\alpha\cap X:A_\gamma\notin\A_\alpha\}.
$$
\item If $\alpha\in Y$ and
$$
p_\alpha\forces
\pi(\dot T_\alpha)\notin C^*(\check\B_\alpha), 
\pi(\dot T_\alpha) \ \hbox{commutes with all elements of} \
C^*(\check\B_\alpha),$$
then there is $p_\alpha'\leq p_\alpha$
such that
$$
p_\alpha'\forces
[\check B_\alpha,\pi(\dot T_\alpha)]\neq0.
$$
\item If $\alpha\in Z$ and
$$
p_\alpha\forces
C^*(\check\B_\alpha)\subseteq\dot\CC_\alpha,
$$
then there are $p_\alpha'\leq p_\alpha$ and a $\PP$-name
$\dot R_\alpha$ such that
$$
p_\alpha'\forces
\dot R_\alpha\in\dot\CC_\alpha
\quad\hbox{and}\quad
[\check B_\alpha,\dot R_\alpha]\neq0.
$$
\end{enumerate}

 In particular, conditions (1) and (2) imply that
$\A_\beta\subseteq\A_\alpha$, $\B_\beta\subseteq\B_\alpha$ and
$h_\beta\subseteq h_\alpha$ whenever $\beta<\alpha<\omega_1$.

Put
$$
\A'=\bigcup_{\alpha<\omega_1}\A_\alpha,\qquad
\B=\bigcup_{\alpha<\omega_1}\B_\alpha,\qquad
h=\bigcup_{\alpha<\omega_1}h_\alpha.
$$
Condition (3) and the fact that $X$ is unbounded in $\omega_1$ imply that
every element $A_\gamma$, for $\gamma\in X$, eventually belongs to some
$\A_\alpha$. Hence $\A'=\A$, and $h:\A\longrightarrow\B$ is a Boolean
isomorphism. This proves (i).

Now let us see that condition (4) implies (ii). Otherwise, by
Lemma \ref{sa-masa}, there are $p\in\PP$ and a $\PP$-name $\dot T$
for a self-adjoint element of $\bb$ such that
$$
p\forces
\pi(\dot T)\notin C^*(\check\B)
\quad\hbox{and}\quad
\pi(\dot T)\ \hbox{commutes with every element of}\ C^*(\check\B).
$$
Replacing $\dot T$ by an equivalent nice name,
 we may assume that $\dot T$ is a nice name. So there is $\alpha\in Y$ such that
$(p_\alpha,\dot T_\alpha)=(p,\dot T)$.  Since
$\B_\alpha\subseteq\B$, the hypothesis of (4) is satisfied at stage
$\alpha$. Hence there is $p_\alpha'\leq p_\alpha$ such that
$$
p_\alpha'\forces
[\check B_\alpha,\pi(\dot T_\alpha)]\neq0.
$$
But $B_\alpha\in\B$, which contradicts the choice of $p$ and $\dot T$.
Thus $\PP$ forces that $C^*(\check\B)$ is a masa.

Condition (5) implies (iii). Indeed, suppose that for some
$\PP$-name $\dot\CC\in\FF$ there is $p\in\PP$ such that
$$
p\forces C^*(\check\B)=\dot\CC.
$$
Choose $\alpha\in Z$ such that
$(p_\alpha,\dot\CC_\alpha)=(p,\dot\CC)$. Then $p_\alpha$ forces that
$C^*(\check\B_\alpha)\subseteq\dot\CC_\alpha$, and so by (5) there
are $p_\alpha'\leq p_\alpha$ and a $\PP$-name $\dot R_\alpha$ such that
$$
p_\alpha'\forces
\dot R_\alpha\in\dot\CC_\alpha
\quad\hbox{and}\quad
[\check B_\alpha,\dot R_\alpha]\neq0.
$$
On the other hand, $B_\alpha\in\B$, and $p_\alpha'$ forces that both
$\check B_\alpha$ and $\dot R_\alpha$ belong to the commutative algebra
$\dot\CC_\alpha=C^*(\check\B)$, which is a contradiction.

So to prove (i) - (iii) we are left with showing how to do the inductive
step $\alpha<\omega_1$ satisfying (1) - (5).

For $\alpha=0$ put $\A_0=\{0,1\}$ and $\B_0=\{0,1\}$, and let
$h_0:\A_0\longrightarrow\B_0$ be the unique Boolean isomorphism.

Suppose that the objects satisfying (1) - (5) have been constructed
before stage $\alpha<\omega_1$. At nonzero limit stages put
$$
h_\alpha=\bigcup_{\beta<\alpha}h_\beta.
$$

First consider the case of $\alpha\in X$. If
$$
\{\gamma\in\alpha\cap X:A_\gamma\notin\A_\alpha\}
$$
is nonempty, let $\beta$ be its least element and put
$A_\alpha'=A_\beta$, as required in (3). If this set is empty, choose
any
$$
A_\alpha'\in\A\setminus\A_\alpha.
$$
Such an element exists because $\A_\alpha$ is countable and
$|\A|=\omega_1$. By Lemma \ref{injective-Q}, $h_\alpha$ can be extended
to a Boolean monomorphism $h_{\alpha+1}$ defined on the Boolean algebra
generated by $\A_\alpha\cup\{A_\alpha'\}$. Put
$$
B_\alpha=h_{\alpha+1}(A_\alpha').
$$
Then (1) - (3) hold, while (4) and (5) are void.

Now suppose that $\alpha\in Y$. If $p_\alpha$ does not force that
$\pi(\dot T_\alpha)$ belongs to
$
\QQ\setminus C^*(\check\B_\alpha)
$
and commutes with the entire $C^*(\check\B_\alpha)$, we proceed as in
the case of $\alpha\in X$, choosing an arbitrary
$A_\alpha'\in\A\setminus\A_\alpha$.

Suppose, therefore, that
$$
p_\alpha\forces
\pi(\dot T_\alpha)\notin C^*(\check\B_\alpha), \ 
\pi(\dot T_\alpha) \ \hbox{commutes with every element of}\ 
C^*(\check\B_\alpha).$$
 Since $p_\alpha$ forces that $\dot T_\alpha$ is self-adjoint,
$p_\alpha$ forces that
$$
C^*(\check\B_\alpha\cup\{\pi(\dot T_\alpha)\})
$$
is a commutative C*-algebra properly extending
$C^*(\check\B_\alpha)$. By Lemma \ref{split-point}, there is a
$\PP$-name $\dot\U_\alpha$ such that $p_\alpha$ forces that
$\dot\U_\alpha$ is an ultrafilter of $\check\B_\alpha$ which is
C*-split by $\dot T_\alpha$.

The hypothesis of the proposition that $\A$ in the extension by $\PP$ has no
countably generated ultrafilter and  Lemma \ref{gdelta-split}  imply that the ultrafilter
$h_\alpha^{-1}[\dot\U_\alpha]$ is split by an element of
$\A\setminus\A_\alpha$. By strengthening $p_\alpha$, we can decide
such an element. Thus there are $s_\alpha\leq p_\alpha$ and
$\gamma\in X$ such that
$$
A_\gamma\in\A\setminus\A_\alpha
$$
and
$$
s_\alpha\forces
\check A_\gamma\ \hbox{splits}\ h_\alpha^{-1}[\dot\U_\alpha].
$$
Put $A_\alpha'=A_\gamma$. By Lemma \ref{injective-Q} extend
$h_\alpha$ to a Boolean monomorphism $h_{\alpha+1}'$ defined on the
Boolean subalgebra of $\A$ generated by $\A_\alpha\cup\{A_\alpha'\}$. Then
$s_\alpha$ forces that $h_{\alpha+1}'(A_\alpha')$ splits
$\dot\U_\alpha$.

Let $\B_{\alpha+1}'$ be the Boolean algebra generated by $\B_\alpha$
and $h_{\alpha+1}'(A_\alpha')$. We can now apply
Lemma \ref{main-lemma} to
$$
\B_\alpha\subseteq\B_{\alpha+1}'
$$
with $p=s_\alpha$, $R=h_{\alpha+1}'(A_\alpha')$ and
$\dot T=\dot T_\alpha$. We obtain a projection $R'\in\QQ$, a Boolean
monomorphism
$$
g:\B_{\alpha+1}'\longrightarrow\QQ
$$
which is the identity on $\B_\alpha$ and satisfies
$g(h_{\alpha+1}'(A_\alpha'))=R'$, and a condition
$p_\alpha'\leq s_\alpha$ which forces that $\check R'$ does not
commute with $\pi(\dot T_\alpha)$.

Put
$$
B_\alpha=R'
\quad\hbox{and}\quad
h_{\alpha+1}=g\circ h_{\alpha+1}'.
$$
The properties of $g$ imply that $h_{\alpha+1}$ extends $h_\alpha$
and $h_{\alpha+1}(A_\alpha')=B_\alpha$. Moreover, $B_\alpha$
commutes with every element of $\B_\alpha$ by Lemma
\ref{main-lemma}(5), so condition (4) follows from Lemma
\ref{main-lemma}(7). Conditions (3) and (5) are void in this case.

Finally, consider the case of $\alpha\in Z$. If $p_\alpha$ does not
force that
$$
C^*(\check\B_\alpha)\subseteq\dot\CC_\alpha,
$$
we proceed as in the case of $\alpha\in X$.

Suppose, therefore, that
$$
p_\alpha\forces
C^*(\check\B_\alpha)\subseteq\dot\CC_\alpha.
$$
The algebra $C^*(\B_\alpha)$ is separable. On the other hand,
Theorem \ref{p-main} implies that $p_\alpha$ forces that the masa
$\dot\CC_\alpha$ is nonseparable, since a separable masa would be
$*$-isomorphic to $C(L)$ for a compact metrizable space $L$, and every
point of $L$ has character at most $\omega<\mathfrak p$. We may therefore  find a $\PP$-name
$\dot R_\alpha$ for a self-adjoint element such that
$$
p_\alpha\forces
\dot R_\alpha\in
\dot\CC_\alpha\setminus C^*(\check\B_\alpha).
$$
Since $p_\alpha$ forces that
$$
C^*(\check\B_\alpha)\cup\{\dot R_\alpha\}\subseteq\dot\CC_\alpha,
$$
it forces that $\dot R_\alpha$ commutes with every element of
$C^*(\check\B_\alpha)$. Choose a $\PP$-name $\dot T_\alpha'$ for a
self-adjoint element of $\bb$ such that
$$
p_\alpha\forces \pi(\dot T_\alpha')=\dot R_\alpha.
$$

We now repeat the construction from the case $\alpha\in Y$, with
$\dot T_\alpha'$ in place of $\dot T_\alpha$. It gives an element
$A_\alpha'\in\A\setminus\A_\alpha$, a projection $B_\alpha$, an
extension $h_{\alpha+1}$ of $h_\alpha$, and a condition
$p_\alpha'\leq p_\alpha$ such that
$$
p_\alpha'\forces
[\check B_\alpha,\dot R_\alpha]\neq0.
$$
Thus condition (5) is satisfied. Conditions (3) and (4) are void in this case.

This completes the inductive construction and proves the existence,
for every fixed family $\FF$ as above, of a Boolean algebra $\B$ satisfying
(i) - (iii).

We now choose $\FF$ so that the resulting masa has no commutative
lift. Since $\PP$ preserves {\sf CH}, it forces that there are exactly
$2^\omega=\omega_1$ masas in $\bb$ by Proposition 12.3.1 of
\cite{ilijas-book}. Hence there is a family
$$
\FF=\{\dot\CC_\alpha:\alpha\in Z\}
$$
of $\PP$-names such that
$$
\PP\forces 
\{\dot\CC_\alpha:\alpha\in Z\}
=
\{\pi[\M]:\M\ \hbox{is a masa of}\ \bb\}.
$$
Apply the preceding construction to this family $\FF$. If
$C^*(\check\B)$ had a commutative lift in the forcing extension, then
by Lemma \ref{masa-masaQQ} it would be equal to $\pi[\D]$ for some
masa $\D$ of $\bb$. This contradicts (iii). Thus $\PP$ forces that
$C^*(\check\B)$ has no commutative lift.

Finally, we construct the family from the statement of the proposition
by induction on $\xi<\omega_2$. Suppose that the Boolean algebras
$\{\X_\eta:\eta<\xi\}$ have already been constructed. Since
$\xi<\omega_2$, we have $|\xi|\leq\omega_1$. Moreover, $\PP$ forces
that the set of unitary elements of $\QQ$ has cardinality
$2^\omega=\omega_1$. Consequently, there is a family $\FF_\xi$ of
cardinality $\omega_1$ of $\PP$-names such that
$$
\PP\forces
\FF_\xi=
\{\pi[\M]:\M\ \hbox{is a masa of}\ \bb\}
\cup
\{UC^*(\check\X_\eta)U^*:
U\ \hbox{is unitary},\ \eta<\xi\}.
$$
Apply the preceding construction with $\FF=\FF_\xi$, and denote the
resulting Boolean algebra by $\X_\xi$. Then $\X_\xi$ is isomorphic to
$\A$, and $\PP$ forces that $C^*(\check\X_\xi)$ is a masa without a
commutative lift. Moreover, it is not unitarily equivalent to
$C^*(\check\X_\eta)$ for any $\eta<\xi$.

The family $\{\X_\xi:\xi<\omega_2\}$ therefore satisfies (a) - (c),
which completes the proof.
\end{proof}

We now pass from the Cohen indestructible masas constructed under {\sf CH}
to Boolean algebras appearing in arbitrary Cohen extensions.

\begin{proposition}[{\sf CH}]\label{cohen-proposition}
Let $\kappa\geq\omega_2$ and suppose that $\dot \A$ is a $\PP_\kappa$-name such that
$
\PP_\kappa$ forces that $\dot \A$ is a Boolean algebra of cardinality
 $\omega_1$  not admitting countably generated  ultrafilters.
Then there are $\PP_\kappa$-names
$\{\dot\CC_\alpha:\alpha<\omega_2\}$ such that $\PP_\kappa$ forces that
\begin{enumerate}
\item[(a)] $\dot\CC_\alpha$ is a masa of $\QQ$ without a commutative lift
for every $\alpha<\omega_2$,
\item[(b)] $\dot\CC_\alpha$ is $*$-isomorphic to $C(K_{\dot \A})$
for every $\alpha<\omega_2$,
\item[(c)] $\dot\CC_\alpha$ and $\dot\CC_\beta$ are unitarily non-equivalent
for every $\alpha<\beta<\omega_2$.
\end{enumerate}
\end{proposition}

\begin{proof}
Since $\dot\A$, together with its Boolean operations, can be coded by a structure
of cardinality $\omega_1$, we may assume that there is $S\subseteq\kappa$ of cardinality $\omega_1$
such that $\dot\A$ is a $\PP_S$-name. By the factorization
of $
\PP_\kappa$  into $\PP_S\times\PP_{\kappa\setminus S}$,
the forcing $\PP_S$ forces that, after the interpretation of $\dot\A$,
the forcing $\PP_{\kappa\setminus S}$ forces that this Boolean algebra
does not admit a countably generated ultrafilter.

Let $\G_S\subseteq\PP_S$ be generic and work temporarily in $V[\G_S]$. Put
$\A$ to be the interpretation of $\dot\A$ in $V[\G_S]$.
Since $|S|=\omega_1$, the forcing $\PP_S$ is c.c.c. and has cardinality
$\omega_1$. Since the ground model satisfies {\sf CH}, we have
${\sf CH}$ in $V[\G_S]$.
Now we note that 
$$
\PP\forces \check\A\ \hbox{does not admit a countably generated ultrafilter}.\leqno (1)
$$
Indeed, otherwise some condition in $\PP$ would force that $\check\A$
has a countably generated ultrafilter. Such an ultrafilter, together with
a countable family generating it, would remain a countably generated
ultrafilter after forcing with $\PP\times\PP_{\kappa\setminus S}$ which is isomorphic to 
$\PP_{\kappa\setminus S}$. This would
contradict the fact that $\PP_{\kappa\setminus S}$ forces that $\check\A$
does not admit a countably generated ultrafilter.

Using (1), we may therefore apply Proposition \ref{ch-proposition} in $V[\G_S]$.
It gives a family
$$
\{\X_\alpha:\alpha<\omega_2\}
$$
of Boolean algebras of projections in $\QQ$ in $V[\G_S]$ such that
\begin{enumerate}
\item[(2)] $\X_\alpha$ is Boolean isomorphic to $\A$
for every $\alpha<\omega_2$,
\item[(3)] $\PP$ forces that $C^*(\check\X_\alpha)$ is a masa of $\QQ$
without a commutative lift for every $\alpha<\omega_2$,
\item[(4)] $\PP$ forces that $C^*(\check\X_\alpha)$ and
$C^*(\check\X_\beta)$ are unitarily non-equivalent
for every $\alpha<\beta<\omega_2$.
\end{enumerate}
The isomorphism between ground model algebras is absolute so
 $\PP_{\kappa\setminus S}$ forces that $\X_\alpha$ is Boolean isomorphic to $\A$
for each $\alpha<\omega_2$.

We will show that the conclusions forced in (3) and (4) are in fact forced by
$\PP_{\kappa\setminus S}$ as well. Let $\G_{\kappa\setminus S}$ be any
$\PP_{\kappa\setminus S}$-generic filter over $V[\G_S]$. By the discussion
in Subsection 2.6, every operator and every code of a masa of $\bb$ appearing
in $V[\G_S\times\G_{\kappa\setminus S}]$ already appears in an intermediate
model obtained by adding Cohen reals indexed by a countable nonempty set
$D\subseteq\kappa\setminus S$. Since $\PP_D$ is isomorphic to $\PP$,
statements (3) and (4) hold in every such intermediate model.

Any failure of (3) or (4) has a witness of this kind. Indeed, by Lemma
\ref{sa-masa}, failure of maximality is witnessed by a self-adjoint element
of $\QQ$ commuting with $C^*(\X_\alpha)$; unitary equivalence is witnessed
by a unitary element of $\QQ$; and, once maximality has been established,
Lemma \ref{masa-masaQQ} shows that the existence of a commutative lift is
witnessed by a masa $\M$ of $\bb$ such that
$\pi[\M]=C^*(\X_\alpha)$. Representatives of the first two witnesses and
a code of the last one belong to one of the above intermediate models, where
the same objects are witnesses, contradicting (3) or (4). Thus both
conclusions are preserved after forcing with $\PP_{\kappa\setminus S}$.

Returning to the ground model, choose $\PP_S$-names $\dot\X_\alpha$
for the algebras $\X_\alpha$, for $\alpha<\omega_2$, and let
$\dot\CC_\alpha$ be a $\PP_\kappa$-name such that
$$
\PP_\kappa\forces \dot\CC_\alpha=C^*(\dot\X_\alpha).
$$
By (2), Lemma \ref{cstar-ba-iso} and the Stone duality, for every $\alpha<\omega_2$
the forcing $\PP_\kappa$
forces that
$\dot\CC_\alpha$ is $*$-isomorphic to $C(K_{\dot\A})$.
Together with the preceding preservation
arguments, this proves (a) - (c).
\end{proof}

\section*{Use of digital assistance}

During the preparation of this paper the author  used 
 a commercial digital assistant, often referred to as artificial intelligence (AI).

It was used for language editing, checking the internal consistency
of notation, references and LaTeX, and discussing the exposition of proofs
and possible gaps. The mathematical ideas and arguments, as well as all
decisions concerning their final form, are the author's. Suggestions produced
by the assistant were checked by the author before being incorporated into
the manuscript. Bibliographic information suggested by the assistant was
independently verified.

The above paragraph was prepared by the digital assistant itself.

\bibliographystyle{amsplain}

\end{document}